\documentclass{article}
\date{12 June 2021}
\usepackage{amssymb,amsthm,amsfonts,graphics}
\usepackage{fullpage}
\usepackage{rotating}
\usepackage{stmaryrd}
\usepackage{tikz}
\usepackage{verbatim}
\usepackage{url}
\usepackage{multicol}
\usepackage{xr-hyper}
\usepackage{mathrsfs}
\usepackage{float}

\theoremstyle{plain}
\newtheorem{theorem}{Theorem}
\newtheorem{lemma}[theorem]{Lemma}
\newtheorem{observation}[theorem]{Observation}
\newtheorem{corollary}[theorem]{Corollary}
\newtheorem{proposition}[theorem]{Proposition}
\newtheorem{conjecture}[theorem]{Conjecture}

\makeatletter
\newtheorem*{rep@theorem}{\rep@title}
\newcommand{\newreptheorem}[2]{%
\newenvironment{rep#1}[1]{%
 \def\rep@title{#2 \ref{##1}}%
 \begin{rep@theorem}}%
 {\end{rep@theorem}}}
\makeatother

\newreptheorem{theorem}{Theorem}
\newreptheorem{lemma}{Lemma}
\newreptheorem{corollary}{Corollary}

\newcommand{\N}{\mathbb{N}}

\usepackage[tbtags]{amsmath}
\usepackage[cal=boondoxo]{mathalfa}

\newcommand{\hered}{\mathrm{Hered}}
\newcommand{\strong}{\mathrm{cHered}}
\newcommand{\minor}{\mathrm{Minor}}
\newcommand{\tminor}{\mathrm{tMinor}}

\newcommand{\crank}{\mathrm{cr}}

\newcommand{\inu}{\in_u}

\newcommand{\pr}{\mathbb{P}}

\newcommand{\cA}{\mathcal A}
\newcommand{\cB}{\mathcal B}
\newcommand{\cC}{\mathcal C}

\newcommand{\cE}{\mathcal E}
\newcommand{\cF}{\mathcal F}
\newcommand{\cG}{\mathcal G}

\newcommand{\cL}{\mathcal L}

\newcommand{\cN}{\mathcal N}
\newcommand{\cO}{\mathcal O}
\newcommand{\cP}{\mathcal P}
\newcommand{\cS}{\mathcal S}

\newcommand{\cX}{\mathcal X}

\newcommand{\cZ}{\mathcal Z}

\newcommand{\tA}{\widetilde{\mathcal A}}
\newcommand{\tB}{\widetilde{\mathcal B}}

\newcommand{\tE}{\widetilde{\mathcal E}}

\newcommand{\tP}{\widetilde{\mathcal P}}

\newcommand{\ext}{{\rm ext}}

\newcommand{\minext}{{\rm minext}}
\newcommand{\xs}{{\rm xs}}

\newcommand{\bS}{\mathbf S}

\newcommand{\bN}{\mathbf N}
\newcommand{\eps}{\varepsilon}

\newcommand{\egmax}{\mathrm eg_{\max}}

\title{Classes of graphs embeddable in order-dependent surfaces}

\author{
Colin McDiarmid\\Department of Statistics\\ University of Oxford\\
cmcd@stats.ox.ac.uk\\
\and
Sophia Saller\\
Department of Mathematics\\ University of Oxford\\
and DFKI\\
sophia.saller@dfki.de}

\begin{document}
\maketitle

\begin{abstract}
Given a 
function $g=g(n)$ we let $\cE^g$ be the class of all graphs $G$ such that if $G$ has order $n$ (that is, has $n$ vertices) then it is embeddable in some surface of Euler genus at most $g(n)$, and let $\tE^g$  be the corresponding class of unlabelled graphs.
We give estimates of the sizes of these classes. For example we show that if $g(n)=o(n/\log^3n)$ then the class $\cE^{g}$ has growth constant $\gamma_{\cP}$, the (labelled) planar graph growth constant; 
and when $g(n) = O(n)$ we estimate the number of n-vertex graphs in $\cE^{g}$ and $\tE^g$ up to a factor exponential in $n$.  From these estimates we see that, if $\cE^g$ has growth constant $\gamma_{\cP}$ then we must have $g(n)=o(n/\log n)$,
and the generating functions for $\cE^g$ and $\tE^g$ have strictly positive radius of convergence if and only if $g(n)=O(n/\log n)$. Such results also hold when we consider orientable and non-orientable surfaces separately. We also investigate related classes of graphs where we insist that, as well as the graph itself, each subgraph is appropriately embeddable (according to its number of vertices); and classes of graphs where we insist that each minor is appropriately embeddable.
In a companion paper~\cite{MSrandom}, 
these results are used to investigate random $n$-vertex graphs sampled uniformly from $\cE^g$ or from similar classes.
\end{abstract}

\section{Introduction}
\label{sec.intro}
Given a surface $S$, let $\cE^S$ be the class of all (finite, simple, labelled) graphs embeddable in $S$ (not necessarily cellularly), so the class $\cP$ of planar graphs is $\cE^{\bS_0}$ where $\bS_0$ is the sphere.  A \emph{genus function} is a function $g=g(n)$ from the positive integers to the non-negative integers: we shall always take $g$ to be such a function. We let $\cE^g$ be the class of all graphs $G$ such that if $G$ has $n$ vertices then $G \in \cE^S$ for some surface $S$ of Euler genus at most $g(n)$. If we insist that all the surfaces involved are {\bf o}rientable we obtain the graph class $\cO\cE^g$,
and similarly if we insist that all the surfaces are {\bf n}on-orientable we obtain $\cN\cE^g$ (where $\cN\cE^0$ is taken to be $\cP$). When $g(n)$ is a constant $h$ for each $n$ we may write $\cE^h$ instead of~$\cE^g$, and similarly for $\cO\cE^h$ and $\cN\cE^h$. For a full discussion of embeddings in a surface see for example~\cite{GraphsonSurfaces}.

The class $\cP$ of planar graphs, and more generally the classes $\cO\cE^h$ and $\cN\cE^h$ of graphs embeddable in a fixed surface, have received much attention recently. In the planar case, much is known 
about the size of such classes as well as about typical properties of graphs in the class, see for example \cite{NumLabCon, outerplanar, GenLabelled, GenOutP, RandCub, PlanGraphsViaWell, TherandomPlanar, QuadraticExact, EdgesRand, RandPlanNnod, RandPlanAvgDeg, LabGC, Asymptoticformula, GCgiven3, MaxDegree, RandomPlanar, addable, RandTriang}.
The corresponding questions for graphs on a fixed general surface have also been extensively studied and much is known, see for example \cite{asymptLabGraphs, LimitLawsFixedS, evolutionRandomS, EvolutionGiant, PhaseTransitions, EnumerationCubicSurf, graphsSurfaces}. 
Given a class $\cA$ of (labelled) graphs we let $\cA_n$ be the set of graphs in $\cA$ on vertex set $[n]=\{1,\ldots,n\}$. 
The class $\cA$ 
has \emph{(labelled) growth constant} $\gamma\,$ if $0<\gamma<\infty$ and
\begin{equation} 
    \left( \left|\cA_n \right| / n! \right)^{\frac1{n}} \rightarrow \gamma \;\; \mbox{ as } n \rightarrow \infty \, .\notag
\end{equation}
The class $\cP$ of planar graphs has growth constant $\gamma_{\cP}\,\approx 27.23\,$ \cite{RandomPlanar, Asymptoticformula}; and for each fixed $h\,$, the class $\cE^h$ has the same growth constant $\gamma_{\cP}$ \cite{graphsSurfaces} (and thus so also have $\cO\cE^h$ and $\cN\cE^h$).
Precise asymptotic estimates are known for the sizes of these classes (when $h$ is fixed), see~(\ref{eqn.limitprobG}) and~(\ref{eqn.limitprobH}) below.

Given a class $\cA$ of (labelled) graphs we let $\tA$ be the corresponding set of unlabelled graphs.  
A set $\tA$ of unlabelled graphs has \emph{unlabelled growth constant} $\tilde{\gamma}\,$ if $0<\tilde{\gamma}<\infty \,$ and
\begin{equation} 
     |\tA_n |^{\frac1{n}} \rightarrow \tilde{\gamma} \;\; \mbox{ as } n \rightarrow \infty \, .\notag
\end{equation} 
For example for outerplanar graphs the unlabelled growth constant is known precisely 
and equals roughly $7.50360$~\cite{BFKV2007}; and
the set $\tP$ of unlabelled planar graphs has unlabelled growth constant $\tilde{\gamma}_{\tP}$ where $\gamma_{\cP} < \tilde{\gamma}_{\tP} \leqslant 32.2$, see~\cite{RandomPlanar}.

We are interested in the case when the genus function value $g(n)$ may grow with $n$, and so the surfaces are not fixed.
At the opposite extreme from $\cP$, when $g(n)$ is very large
all graphs are in $\cE^g$
(when $g(n)$ is at least about $\tfrac16 n^2$, see near the end of Section~\ref{subsec.embed} for precise values).
In the overarching project we investigate two closely related questions for a given genus function $g=g(n)\,$: (a) how large are the graph classes $\cE^g$, $\cO\cE^g$ and $\cN\cE^g$;
and (b) what are typical properties of a random $n$-vertex graph $R_n$ sampled uniformly from such a class?
We also consider unlabelled graphs, more briefly.
In the present paper we consider question (a), and we give estimates and bounds on the sizes of these classes of graphs (and of related more constrained classes of graphs - see the next section).  In a companion paper~\cite{MSrandom}, we use these results in investigations of question (b) concerning random graphs. A central aim in both of these papers is to find 
where there is a change
between `planar-like' behaviour and behaviour like that of an Erd\H{o}s-R\'enyi (binomial) random graph, both for class size and for typical properties.  It seems that this `phase transition' occurs when $g(n)$ is around $n/\log n$.
See~\cite{DKMS2019} for results on the evolution of random graphs on non-constant orientable surfaces when we consider also the number of edges.


\section{Statement of Results}
We first consider classes of graphs which are embeddable in given surfaces, where we insist simply that the graph be embeddable in the appropriate surface (of Euler genus $g(n)$ for an $n$-vertex graph) and we have no other requirements.  Our focus is mostly on this case, presented in
Section~\ref{subsec.results1}.
In Section~\ref{subsec.results2} we consider classes of graphs which are `hereditarily embeddable' in given surfaces, where we insist also that each  induced subgraph is embeddable in an appropriate surface, depending on its number of vertices.
Then, in Section~\ref{subsec.results3}, we consider classes of graphs where we insist that each minor is appropriately embeddable.
Finally, we close this section with a brief plan of the rest of the paper.

Throughout this paper, $g=g(n)$ is a genus function and $\cA^g$ denotes any one of the graph classes $\cO\cE^g$, $\cN\cE^g$, $ \cE^g \, (= \cO\cE^g \cup \cN\cE^g)$ or $\cO\cE^g \cap \cN\cE^g$. 
Given non-negative functions $x(n)$ and $y(n)$ for $n \in \N$, the notation $x(n) \ll y(n)$ means that $x(n)/y(n) \to 0$ as $n \to \infty$.
We also use the standard notations
$o(x(n))$, $O(x(n))$ and $\Theta(x(n))$, always referring to behaviour as $n \to \infty$.

\subsection{Classes $\cA^g$ of graphs embeddable in given surfaces}
\label{subsec.results1}
We present three theorems (and two corollaries) in this section.
The first theorem gives estimates of the size of the set $\cA^g_n$ of graphs for `small' genus functions $g$, and is our main result since it covers the `phase transition' range for $g$.  The second and third theorems give lower bounds (and some estimates) and then upper bounds on the size of $\cA^g_n$ for wider ranges of the genus function $g$.
By convention, if $t=0$ then both $t^t$ and $(1/t)^t$ mean~1.
Recall that $\gamma_{\mathcal{P}}$ is the labelled planar graph growth constant.

\begin{theorem} \label{thm.gc-estimate}
(a) If $g(n)$ is $o(n / \log^3n)$, then $\cA^{g}$ has growth constant $\gamma_{\mathcal{P}}$; that is,
\begin{equation}
\left|\cA_n^{g}\right| = (1+o(1))^n \, \gamma_{\cP}^n \; n! \, .\notag
\end{equation}
(b)
If $g(n)$ is $O(n)$, then
\begin{equation} 
\left|\cA_n^{g}\right| = 2^{\Theta(n)} \, g^{g} \; n! \;\;\; \mbox{ and } \;\;\; |\tA_n^{g}| = 2^{\Theta(n)} \, g^{g} \, .\notag
\end{equation}
\end{theorem}
Since $\cP \subseteq \cA^g \subseteq \cE^g$, in part (a) it would suffice to take $\cA=\cE$.  We have no result for unlabelled graphs corresponding to part (a). 
Note that in the equations in part (b) above we write $g$ rather than $g(n)$ for readability - we shall often do this.

For a class $\cB$ of (labelled) graphs, we let $\rho(\cB)$ be the radius of convergence of the exponential generating function $B(x) = \sum_n |\cB_n|/n! \: x^n$, so $\rho(\cB) = \left(  \limsup_{n\rightarrow \infty} \left( |\cB_n| / n! \right)^{\frac1{n}}\right)^{-1}$.  Thus $0 \leqslant \rho(\cB) \leqslant \infty$, and for example $\rho(\cP)= \gamma_{\cP}^{-1}$. 
Similarly, for a set $\tB$ of unlabelled graphs, we let $\tilde{\rho}(\tB)$ be the radius of convergence of the ordinary generating function $\tilde{B}(x) = \sum_n |\tB_n| \: x^n$, so $\tilde{\rho}(\tB) = \left(  \limsup_{n\rightarrow \infty} |\tB_n|^{\frac1{n}}\right)^{-1}$.
Thus $0 \leqslant \tilde{\rho}(\tB) \leqslant \infty$, and for example $\tilde{\rho}(\tP)= \tilde{\gamma}_{\tP}^{-1}$.
Observe that by 
Theorem~\ref{thm.gc-estimate} (b)
\begin{equation}\label{eqn.rhopos}
\rho(\cA^g)>0 \mbox{ if and only if } g(n) = O(n/\log n)
\end{equation}
and
\begin{equation} \label{eqn.rhotpos}
\tilde{\rho}(\tA^g)>0 \mbox{ if and only if } g(n) = O(n/\log n).
\end{equation}
Thus, in both the labelled case $\cA^g$ and the unlabelled case $\tA^g$, the threshold when the radius of convergence drops to 0 is when $g(n)$ is around $n/\log n$.
\smallskip

We next give two theorems yielding lower bounds (Theorems~\ref{thm.lowerboundab} and~\ref{thm.lowerbound}) and one theorem yielding upper bounds (Theorem~\ref{thm.upperbound}) on the sizes of the sets $\cA^g_n$ of graphs, for a wider range of genus functions $g$ than considered in Theorem~\ref{thm.gc-estimate}. Theorem~\ref{thm.gc-estimate} (b) will follow from the lower bounds in Theorem~\ref{thm.lowerboundab} (b)
(as spelled out in Corollary~\ref{cor.lb}) and the upper bounds in Theorem~\ref{thm.upperbound}. (Theorem~\ref{thm.gc-estimate} (a) will be proved separately).
%
We are most interested in the embeddable class of graphs $\cA^g$, but the lower bounds in Theorem~\ref{thm.lowerboundab} 
apply to the smaller class of graphs which are `freely embeddable'.
Given a genus function $g$, we let $\cF^g$ be the class of graphs $G$ such that \emph{every} embedding system for $G$ has Euler genus at most $g(n)$ where $v(G)=n$.
The freely embeddable class $\cF^g$ of course satisfies $\cF^g \subseteq \cA^g$, and $\cF^g$ may be much smaller than $\cA^g$: for example if $g$ is identically~0 then $\cA^g$ is $\cP$ and
$\cF^g$ is
the class of forests.

The lower bound in part (a) of Theorem~\ref{thm.lowerboundab} is for $g(n)=o(n)$ and lets us relate $|\cA^g_n|$ to $|\cP_n|$, whilst the lower bound in part (b) is for \emph{all} genus values $h$. Recall that always $\left|\cA_n^{g}\right| \geqslant \left|\cF_n^{g}\right|$.
\begin{theorem}\label{thm.lowerboundab}
\begin{description}
\item{(a)}  If $g(n)$ is $o(n)$ then 
\begin{equation} 
\left|\cF_n^{g}\right| \geqslant  (1+o(1))^{n} \; \gamma_{\cP}^n \; n! \; g^{g/2}.\notag
\end{equation}
\item{(b)} There is a constant $c>0$ such that, for every $h \geqslant 0$ and $n \geqslant 1$, 
\begin{equation} 
\left|\cF_n^{h}\right| \geqslant c^{n+h} \,(n^2/h)^h \: n! \,.\notag
\end{equation}
\end{description}
\end{theorem}
It follows from Theorem~\ref{thm.lowerboundab} (a) (by considering for example the function $\min\{g(n),n/\log n\}$), that if $\limsup_{n\rightarrow \infty}g(n)\, \tfrac{\log n}{n} >0$ then $ \limsup_{n\rightarrow \infty} \left( |\cF_n^g| / n! \right)^{\frac1{n}}\!
> \gamma_{\cP}$ and so $\rho(\cA^g) \leqslant \rho(\cF^g) < \rho(\cP)$.
Thus if $\cA^g$ or indeed $\cF^g$ has growth constant $\gamma_{\cP}$ then we must have $g(n)=o(n/\log n)$.
In 
Theorem~\ref{thm.lowerboundab} (b), 
the constant $c>0$ 
need not be tiny: 
the proof will show that if we restrict our attention to $n \geqslant 15$ (and any $h \geqslant 0$) then we may take $c=\tfrac13$, see inequality~(\ref{eqn.Alb}).
(Recall that if $h=0$ then $(n^2/h)^h$ is taken to be 1.)
%
If we restrict our attention to values $h$ which are at most linear in $n$ we  obtain the following corollary.
\begin{corollary} \label{cor.lb}
Given $c_0>0$ there exists $c > 0$ such that
if $\, 0 \leqslant h \leqslant c_0 \, n\,$ then
\begin{equation} \label{eqn.lb-smallg} \left|\cA_n^h \right| \geqslant \left|\cF_n^h \right| \geqslant c^n \, h^h \: n!  \;\;\; \mbox{ and thus } \;\;\; \left|\tA_n^h\right|\geqslant 
\left|\widetilde{\cF}_n^h \right|
\geqslant c^n \, h^h \,\notag \text{ .}
\end{equation}
\end{corollary}

%

Theorem~\ref{thm.lowerbound} gives some lower bounds on $|\cA^g_n|$ (not on $|\cF^g_n|$) when $g$ is large, and some estimates when $g$ is very large.  
The lower bound in part~(a) strengthens the lower bound on $|\cA^h_n|$ yielded by Theorem~\ref{thm.lowerboundab} (b) in some cases when $g(n) \geqslant n^{1+\delta}$ for some $\delta>0$; and in part (b), when $g$ is very large, 
we obtain asymptotic estimates of $|\cA^g_n|$.


\begin{theorem}\label{thm.lowerbound}
\begin{description}
\item{(a)}
If $j \in \N$ is fixed and 
$\,n^{1+1/(j+1)}\ll g(n) \ll n^{1+1/j}$, then
\begin{equation}
\left|\cA_n^g \right| \geqslant \,(n^2/g)^{(1+o(1))\frac{j+2}{j} g}\,.\notag
\end{equation}
%
%
\item{(b)}
If $g(n) \gg n^{3/2}$ and
$\bar{g}(n) = \min\{g(n), \lfloor \tfrac1{12}n^2 \rfloor\}$, then
\begin{equation} \label{eqn.gbig1}
 |\cA^g_n| = \binom{\binom{n}{2}}{3\bar{g}}^{1+o(1)}\,.
\end{equation}
Thus
\begin{equation} \label{eqn.gbig2}
|\cA_n^g| = (n^2 / g)^{(3+o(1))\, g} \;\;\: \mbox{ if } \; n^{3/2} \ll g(n) \ll n^2 \,,
\end{equation}
and 
\begin{equation} \label{eqn.gbig3}
|\cA^g_n| = 2^{(\tfrac12 +o(1))\, H(6c)\, n^2} \;\;\: \mbox{ if } \; g(n) \sim c n^2 \mbox{ for some } \, 0< c \leqslant \tfrac1{12} \,.
\end{equation}
\end{description}
\end{theorem}
%
The lower bound in part (a) does not hold for the freely embeddable class $\cF^g_n$. 
Indeed we shall see as a corollary of a fuller and more precise result (Proposition~\ref{thm.estF} in Section~\ref{subsec.estFhn}) that
\begin{equation} \label{eqn.largegF2}
    |\cF^g_n| =  (n^2/g)^{(1+o(1))g} \;\;\; \mbox{ if }\; n \ll g(n) \ll n^2.
\end{equation}
%
Observe that equation~(\ref{eqn.gbig2}) in part (b) of Theorem~\ref{thm.lowerbound} shows that we have approximate equality in the case $j=1$ of part~(a).
Recall that the entropy function $H(p)$ which appears in equation~(\ref{eqn.gbig3}) is given by $H(p) = - p \log_2 p - (1-p) \log_2(1-p)$ for $0 \leqslant p \leqslant 1$, and that $H(\tfrac12)=1$.
(When $\log$ has no subscript it means natural log.) 
Thus, by equation~(\ref{eqn.gbig3}), 
if $g(n) \geqslant \tfrac1{12} n^2$ then
$|\cA^g_n| = 2^{(1+o(1)) \binom{n}2}$.
For comparison, note that if $g(n) \geqslant \tfrac16n^2$ then {\bf all} graphs are in $\cA^g$ (that is, all $2^{\binom{n}{2}}$ graphs on $[n]$ for each $n$), and if $g(n) \geqslant \tfrac12n^2$ then all graphs are in $\cF^g$ -- see Section~\ref{subsec.embed} below.


\smallskip

Our last theorem in this subsection gives upper bounds on $|\cA_n^{h}|$ and $|\tA_n^{h}|$.

\begin{theorem}\label{thm.upperbound}
There is a constant $c$ such that, for every $h \geqslant 0$ and $n \geqslant 1$, 
\begin{equation} 
|\tA_n^{h}|\leqslant c^{n+h} \, h^{h} \;\;\; \mbox{ and thus } \;\;\; \left|\cA_n^{h}\right|\leqslant c^{n+h} \, h^{h} \, n! \, .\notag
\end{equation}
\end{theorem}

Theorem~\ref{thm.lowerbound} (d) gives the estimates~(\ref{eqn.gbig2}) and~(\ref{eqn.gbig3}) for $|\cA^g_n|$ when $g$ is very large. Theorem~\ref{thm.upperbound} together with Theorem~\ref{thm.lowerbound} (b) and (c) will allow us to give estimates of $|\cA^g_n|$ for certain other genus functions $g(n) \gg n$.
%
\begin{corollary}
\label{corollary_gg}
Suppose that either $\eta=0$ or 
$\eta = \tfrac1{j+1}$ for some integer $j \geqslant 1$, and let $g(n)=n^{1+\eta+o(1)}$ with $g(n) \gg n^{1+\eta}$.  Then
\[ |\cA^g_n| = g^{(1+o(1)) g}\,.\]
\end{corollary} 

\subsection{Hereditary classes of graphs, where each subgraph embeds appropriately}
\label{subsec.results2}

Our definition of the graph class $\cA^g$ treats each number $n$ of vertices completely separately, but we might wish to be more demanding and insist for example that each subgraph embeds in the appropriate surface, and thus the corresponding class is closed under forming subgraphs.  Since the appropriate surface is determined by the number of vertices, this is equivalent to insisting that the class is closed under forming induced subgraphs, that is, the class is \emph{hereditary}.

Given a graph class $\cB$, we say that a graph $G$ is \emph{hereditarily} in $\cB$ if for each nonempty set $W$ of vertices the induced subgraph $G[W]$ is in  $\cB$ ; and we let $\hered(\cB)$ be the class of graphs which are hereditarily in~$\cB$. Observe that the class $\hered(\cB)$ is hereditary: we call it the \emph{hereditary part} of $\cB$. Given a genus function $g$ 
we are interested here in $\hered(\cA^g)$.
Since $\cP\subseteq \hered(\cA^g) \subseteq \cA^g$, 
Theorem \ref{thm.gc-estimate}(a)
shows that $\hered(\cA^g)$ has growth constant 
$\gamma_{\cP}$ as long as $g(n) = o(n/\log^3\!n)$.  

We give an upper bound (in Proposition~\ref{prop.heredsmall}) then a lower bound (in Theorem~\ref{cor.excess}) 
on $|\hered(\cA^g_n)|$.
For many genus functions $g$ which `often increase', $\hered(\cA^g_n)$ is much smaller than $\cA^g_n$, as shown in the following result (where the value of $\alpha$ is not optimised).
\begin{proposition} \label{prop.heredsmall}
Let the genus function $g$ satisfy $g(n) =o(n/\log^3n)$; and suppose that there is an $n_0$ such that for all $n \geqslant n_0$, $g(n) > g(n-k)$ for some $1 \leqslant k \leqslant \alpha n$, where $\alpha= \frac16$.
Then $|\hered(\cA^g_n)| \ll |\cA^g_n|$.
\end{proposition}
\noindent
Examples of genus functions $g$ as in this proposition include the round up or down of $\beta \log n$ for large $\beta$,  $n^{\beta}$ for $0<\beta<1$, and $n \log^{-\beta}n$ for $\beta>3$.
We now consider larger genus functions $g$.  Recall that always $\cA^g \supseteq \cF^g$, so $\hered(\cA^g) \supseteq \hered(\cF^g)$ : thus the next result gives a lower bound on $|\hered(\cA^g)_n|$.

\begin{theorem} \label{cor.excess}
If 
$g(n) \gg n/\log n$ then $\,\left(|\hered(\cF^g)_n|/n! \right)^{1/n} \to \infty$ as $n \to \infty$.
\end{theorem}
Let us revisit the results~(\ref{thm.gc-estimate}) and~(\ref{eqn.rhotpos}) above.  By definition $\cA^g \supseteq \cF^g \supseteq \hered(\cF^g)$, and we now see that
\begin{equation}\label{eqn.rhopos-hered}
\rho(\cA^g)>0 \,\mbox{ if }\, g(n) = O(n/\log n) \;\;\; \mbox{ and } \;\;\;
\rho(\hered(\cF^g))=0 \,\mbox{ if }\, g(n) \gg n/\log n \,.
\end{equation}
Thus we see that, despite considering a worst possible embedding and the additional hereditary constraint,
the threshold when the radius of convergence of $\hered(\cF^g)$ drops to zero still occurs when $g(n)$ is around $n/\log n$, as for the embeddable case $\cA^g$.
Similarly for unlabelled graphs
\begin{equation}\label{eqn.rhotpos-hered}
\tilde{\rho}(\widetilde{\cA^g})>0 \, \mbox{ if }\, g(n) = O(n/\log n) \;\;\; \mbox{ and } \;\;\;
\tilde{\rho}(\hered(\widetilde{\cF^g}))=0 \,\mbox{ if }\, g(n) \gg n/\log n .
\end{equation}

%
\smallskip

We could be even more demanding than above, where we require that each induced subgraph has a suitable embedding. We could insist that we can choose one embedding $\phi$ of the original graph $G$, and then use the induced embedding for each induced subgraph of $G$, so that $\phi$ `certifies' that $G \in \hered(\cA^g)$.  See~Section~\ref{subsec.stronghered} where we consider such `certifiably hereditarily embeddable' graphs.


\subsection{Minor-closed classes of graphs, where each minor embeds appropriately}

\label{subsec.results3}

Let us now insist that each minor of our graphs (rather than each induced subgraph) is appropriately embeddable.  Recall that a graph $H$ is a \emph{minor} of a graph $G$ if $H$ can be obtained from a subgraph of $G$ by a sequence of edge-contractions, see for example~\cite{BondyMurty,Diestel}.
Given a class $\cB$ of graphs, let $\minor(\cB)$ be the class of graphs $G$ such that each minor of $G$ is in $\cB$. Thus $\minor(\cB)$ is minor-closed: we call it the \emph{minor-closed part} of $\cB$ (which is the same as the minor-closed part of $\hered(\cA^g)$).
Of course we always have $\cP \subseteq \minor(\cA^g) \subseteq \cA^g$, and so in particular
$\rho(\cP) \geqslant \rho(\minor(\cA^g))$.
Also by the definitions we always have
$\rho(\minor(\cA^g)) \geqslant \tilde{\rho}(\minor(\tilde{\cA}^g))$.
We give one theorem concerning $\minor(\cA^g)$, with contrasting parts.  Note that there is no hint here of a change in behaviour when $g(n)$ is around $n/\log n$.
\begin{theorem}\label{thm.minor}
For every genus function $g$, either $\minor(\cA^g)$ contains all graphs, or $\tilde{\rho}(\minor(\tilde{\cA}^g)) >0$ (and so $\rho(\minor(\cA^g))>0$).
For every $\eps>0$ there is a constant $c$ such that if $g(n) \geqslant cn$ then $\rho(\minor(\cA^g)) < \eps$.
\end{theorem}
%
The first part of this theorem shows that if say $g_0(n) \sim \tfrac17 n^2$, so $\minor(\cA^{g_0})$ does not contain all graphs, then $\rho(\minor(\cA^{g_0}))>0$. The second part shows that if $g_1(n)=cn$ for some suitably large constant $c$, then $\rho(\minor(\cA^{g_1})) < \rho(\minor(\cA^{g_0}))$.
This may at first sight seem paradoxical, until we realise that it is not just values of $g(n)$ for large $n$ that matter here. Note that, 
much as in the hereditary case, the graph class $\minor(\cA^g)$ has a growth constant 
$\gamma_{\cP}$ when $g(n)=o(n/\log^3 n)$. 


\subsection{Plan of the paper}
\label{subsec.plan}

We have just presented our main results. The plan of the rest of the paper is as follows. In the next two sections we give some background on embeddings, and then give some preliminary results on how the numbers of graphs in the classes grow when we add a new vertex to the graphs or add a handle to the surface. 
In the following three sections we prove the results stated in Section~\ref{subsec.results1} on classes of graphs embeddable in given surfaces, proving lower bounds (including Theorem~\ref{thm.lowerbound}) in Section~\ref{sec.lb}, proving upper bounds (including Theorem~\ref{thm.upperbound}) in Section~\ref{sec.ub}, and proving  Theorem~\ref{thm.gc-estimate}
in Section~\ref{sec.proofthm1}.  In Section~\ref{sec.hered} we investigate the hereditary class $\hered(\cA^g)$ of hereditarily embeddable graphs discussed in Section~\ref{subsec.results2}, and prove Proposition~\ref{prop.heredsmall} and Theorem~\ref{cor.excess}; and we also investigate the related subclass of `certifiably
In Section~\ref{sec.mc} we consider the minor-closed class $\minor(\cA^g)$ where each minor is appropriately embeddable, discussed in Section~\ref{subsec.results3}, and prove Theorem~\ref{thm.minor}; and we also briefly consider what happens if we replace `minor' by `topological minor'.
Finally, Section~\ref{sec.concl} contains a few concluding remarks and questions.


\section{Some background on embeddings of graphs in surfaces}
\label{sec.back}

In this section we fill in more details of known results on the sizes of the sets $\cO\cE_n^h$ and $\cN\cE_n^h$ of graphs, and then give some background results on embeddings of graphs in surfaces. 

\subsection{Number of graphs embeddable in a fixed surface}
\label{subsec.fixedsurface}

We noted that the class $\cP$ has growth constant $\gamma_{\cP}$ \cite{RandomPlanar}; and further both $\cO\cE^h$ and $\cN\cE^h$ have the same growth constant $\gamma_{\cP}$ for each fixed $h\,$ \cite{graphsSurfaces}.
Gim\'{e}nez and Noy~\cite{Asymptoticformula} give an explicit analytic expression for $\gamma_{\cP}$, showing that $\gamma_{\cP}\approx 27.2269$
(where $\approx$ means `correct to all figures shown').  Also, we have precise asymptotic estimates  \cite{Asymptoticformula}, \cite{asymptLabGraphs}, \cite{LimitLawsFixedS} for the sizes of these classes: for all fixed even $h \geqslant 0$, 
\begin{equation} \label{eqn.limitprobG}
    \left|\cO\cE_n^{h}\right| \sim c^{(h)} \, n^{\frac{5(h-2)}{4}-1}\, \gamma_{\cP}^n \, n! \;\;\;  \text{ as } n \rightarrow \infty
\end{equation}
where $c^{(h)}$ is a positive constant; and for all fixed $h \geqslant 0$,
\begin{equation} \label{eqn.limitprobH}
    \left|\cN\cE_n^{h}\right| \sim \bar{c}^{(h)} \, n^{\frac{5(h-2)}{4}-1}\, \gamma_{\cP}^n \, n! \;\;\;  \text{ as } n \rightarrow \infty
\end{equation}
where $\bar{c}^{(h)}$ is a positive constant.

\subsection{Embeddings of graphs in surfaces} \label{subsec.embed}

We now collect a few useful facts about embeddings of graphs in surfaces which we will use in the remainder of this paper. For a much fuller introduction to graphs on surfaces we refer the 
reader to \cite{GraphsonSurfaces}.
We will always let $h$ be a non-negative integer.  If $h$ is even, $\bS_{h/2}$ denotes the sphere with $h/2$ handles, which is the orientable surface with Euler genus $h$.  We denote the non-orientable surface with Euler genus $h$ by $\bN_h$ for each $h$, where by convention $\bN_0$ means the sphere $\bS_0$ (which is treated also as non-orientable). If a connected graph $G$ has an embedding in $\bS_{h/2}$ then it has a cellular embedding in $\bS_{h'/2}$ for some even $h' \leqslant h$;  and similarly
if $G$ has an embedding in $\bN_{h}$ then it has a cellular embedding in $\bN_{h'}$ for some $h' \leqslant h$.

A key result is Euler's formula.
Recall that we are interested in simple graphs, but it is convenient here to work with \emph{pseudographs}, which may have multiple edges and loops. Let the connected pseudograph $G$ with $v$ vertices and $e$ edges be cellularly embedded in a surface of Euler genus $h$, with $f$ faces. {\bf Euler's formula} states that
\begin{equation} \label{eqn.Euler}
v-e+f = 2-h \, .
\end{equation}
Now suppose that the pseudograph $G$ has $\kappa \geqslant 2$ components $H_1,\ldots,H_{\kappa}$.  If each component $H_i$ has a cellular embedding $\phi_i$ with $f_i$ faces and Euler genus $h_i$ then we say that $G$ has a cellular embedding $\phi$ with $f= \sum_i(f_i -1) +1 = \sum_i f_i -(\kappa-1)$ faces (we think of the `outer faces' of the $\kappa$ embeddings $\phi_i$ as being merged) and  Euler genus $h = \sum_i h_i$. The embedding $\phi$ is orientable if and only if each $\phi_i$ is orientable.
Corresponding to (\ref{eqn.Euler}), Euler's formula for graphs with $\kappa$ components is
\begin{equation} \label{eqn.Euler2}
v-e+f-\kappa  = 1-h \, .
\end{equation}
%
\smallskip

We will sometimes make use of rotation systems or more generally of embedding schemes. We give a very brief introduction here, and refer the reader to Chapter 3 of \cite{GraphsonSurfaces} for a full introduction.
%
Given a pseudograph $G$, for each vertex $v$ let $\pi_v$ be a cyclic permutation of the edges incident to $v$.  We call the family $\pi = \{\pi_v \mid v \in V(G)\}$ a \emph{rotation system} for $G$.
If $G$ is cellularly embedded in an orientable surface then the clockwise ordering around each vertex gives a rotation system for $G$; and conversely a rotation system for $G$ gives a cellular embedding of $G$ in an orientable surface.
A mapping $\lambda: E(G) \rightarrow \{+1, -1\}$ is called a \emph{signature} for $G$.  If $G$ is cellularly embedded in a non-orientable surface then we set $\lambda(e)=1$ if the `clockwise' orderings at the end-vertices of $e$ agree, and $\lambda(e)=-1$ otherwise.  Thus we may obtain an \emph{embedding scheme} $(\pi,\lambda)$ consisting of a rotation system and a signature.
Conversely, an embedding scheme for $G$ gives a cellular embedding of $G$ in a surface $S$, where $S$ is orientable if and only if each cycle has an even number of edges $e$ with $\lambda(e)=-1$. 
\medskip

The \emph{cycle rank} $\crank(G)$ 
of $G$ is $e-v+\kappa$. Observe that $\crank(G) \geqslant 0$, 
and $\crank(G)=0$ if and only if $G$ is a forest.  The cycle rank has several other names, including circuit rank, corank, nullity, cyclomatic number and first Betti number, see for example Bollob\'{a}s~\cite{BelaMGT} and Bondy and Murty~\cite{BondyMurty}.

Given a pseudograph $G$, we let $\egmax(G)$ be the maximum over all embedding systems for $G$ of the Euler genus of the corresponding surface, in which $G$ has a cellular embedding.  We call $\egmax(G)$ the \emph{maximum Euler genus} of $G$.
By Euler's formula~(\ref{eqn.Euler2}),
the Euler genus of a cellular embedding with $f$ faces is $e-v-f+\kappa+1 \leqslant e-v+ \kappa= \crank(G)\,$ (since $f \geqslant 1$); and thus $\egmax(G) \leqslant \crank(G)$.
In fact equality holds here: by a result of Ringel and Stahl, see Theorem 4.5.1 of~\cite{GraphsonSurfaces}, for every pseudograph $G$ the maximum Euler genus equals the cycle rank\,: that is,
\begin{equation} \label{eqn.cr}
 \egmax(G) = \crank(G)\,. 
\end{equation}
Thus 
$\cF^g$ is the class of graphs $G$ with $\crank(G) \leqslant g(n)$ where $n=v(G)$.

If the graph $G$ is embeddable in a surface of Euler genus $h$, then $G$ is cellularly embeddable in a surface of Euler genus $k$ for some $k$ with $0 \leqslant k \leqslant h$.
Since $3f\leqslant 2e$ for all embeddings of simple graphs, from Euler's formula~(\ref{eqn.Euler}) or~(\ref{eqn.Euler2}) we see that
\begin{equation}
  e(G) \leqslant 3(n+h-2) \;\; \mbox{ for each } G \in \cE^h.
\end{equation}

Any pseudograph always has a cellular embedding in some orientable surface and in some non-orientable surface (recall that we treat the sphere $\bS_0$ as both an orientable and a non-orientable surface).
In proofs we will sometimes treat the orientable and non-orientable cases separately. The following observation shows that always $\cO\cE_n^h$ is no bigger than $\cN\cE_n^{h+1}$.
\begin{observation}\label{nonor_from_or}
For each $h \geqslant 0$, a graph $G$ embeddable in any surface of Euler genus $h$ can be cellularly embedded in a non-orientable surface of Euler genus at most $h+1$, so $\cE^h \subseteq \cN\cE^{h+1}$.
\end{observation}
\noindent
This observation is clearly correct if $G$ is acyclic (by the convention that $\bS_0$ is counted also as non-orientable). For any graph $G \in \cO\cE^h$ with a cycle, we may start with a rotation system giving an orientable cellular embedding $\phi$ with Euler genus $h' \leqslant h$, pick an edge $e$ in a cycle, and give $e$ signature -1 (with all other edges having signature +1). We obtain a non-orientable cellular embedding with at most one less face than $\phi$, and so with Euler genus at most $h'+1$.

An example where we need the extra $1$ is the complete graph $K_7$ on seven vertices, which is in $\cO\cE^2$ but not in $\cN\cE^2$. There is no result like  Observation~\ref{nonor_from_or} for orientable surfaces, since for all $h \geqslant 1$ there are graphs in $\cN\cE^1$ but not in $\cO\cE^h$ \cite{projective_graphs}.
By the Ringel-Youngs Theorem (see equation (7) in \cite{CompleteGraph}, or see for example the book \cite{GraphsonSurfaces}, Theorems 4.4.5 and 4.4.6) the maximum Euler genus of a graph on $n$ vertices, that is the Euler genus of the complete graph on $n$ vertices, is equal to $2\lceil \tfrac1{12} (n-3)(n-4)\rceil \sim \tfrac16 n^2$ in the orientable case, and $\lceil \tfrac16 (n-3)(n-4) \rceil \sim \tfrac16 n^2$ in the non-orientable case (apart from when $n=7$ when the value is 3).  These values are actually at most $\frac16 n^2$, so if $g(n) \geqslant \frac16 n^2$ for each $n \in \N$ then $\cA^g$ contains all graphs (and we cannot replace $\tfrac16$ by any smaller constant).

Recall that $\cF^g$ is the class of graphs such that $\egmax(G) \leqslant g(n)$, where $v(G)=n$. 
For a (simple) graph $G$ on $[n]$,
\[ \egmax(G) \leqslant \egmax(K_n) = \binom{n}{2}-n+1 \leqslant \tfrac12 n^2\,,\]
so if $g(n) \geqslant \frac12 n^2$ for each $n \in \N$ then $\cF^g$ contains all graphs (and we cannot replace $\tfrac12$ by any smaller constant).

%



\section{Growth ratios for $\cA^g$ when adding a vertex or handle}
\label{subsec.gr}
In this section we investigate how numbers of graphs embeddable in surfaces grow when we add a vertex to the graph or a handle to the surface.
We give lower bounds on the growth ratio $|\cO\cE_{n+1}^h|/|\cO\cE_n^h|$ when we increment $n$ by 1, and on the growth ratio $|\cO\cE_{n}^{h+2}|/|\cO\cE_n^h|$ when we increment $h$ by 2; and on similar ratios for non-orientable surfaces. (Simultaneous increments are considered in~\cite{thesis}, see Lemma 76).
\medskip

\subsection{Growth ratios when adding a vertex}
\label{gr_vertex}

We first consider incrementing $n$ by 1.  Let us start by noting that, by equations~(\ref{eqn.limitprobG}) and~(\ref{eqn.limitprobH}), for each fixed surface $S$ we have
\begin{equation}\label{incrementingn}
\frac{|\cE^S_{n+1}|}{|\cE^S_n|} \sim \gamma_{\cP}\, n \;\;\mbox{ as }\, n \to \infty.
\end{equation} 

For $n \in \N$ let $\minext(n)$ be the minimum over all graphs $G$ on $[n]$ of the number of graphs $G'$ on $[n+1]$ such that (a) $G'$ restricted to $[n]$ is $G$, and (b) for every surface $S$, if $G$ embeds in $S$ then $G'$ also embeds in~$S$.
Then for every $h\in\mathbb{N}_0$ and $n \in \N$
\begin{equation} \label{eqn.grs}
  |\cA^h_{n+1}| \geqslant \minext(n)\, |\cA^h_n|. \end{equation}
It is not hard to see that 
\begin{equation} \label{eqn.minext}
\minext(n) \geq 2n  \mbox{ for every } n \in \N\,.
 \end{equation}
To show this, let $G$ be a graph on $[n]$.  We may assume wlog that $G$ is connected.  In $G'$, we can make the new vertex $n+1$ be isolated, or be a leaf, or be adjacent to both ends of an edge of $G$.  This gives $1+n+e(G)  \geq 2n$ distinct graphs $G'$; and equation~(\ref{eqn.minext}) follows.

Observe that by inequalities~(\ref{eqn.grs}) and (\ref{eqn.minext}), for every $h\in\mathbb{N}_0$ and every $n \in \N$
\begin{equation} \label{eqn.growth}
  |\cA^h_{n+1}|/ |\cA^h_{n}| \; \geqslant 2n \,.
\end{equation} 
Inequality~(\ref{eqn.growth}) will suffice for our present purposes (in the proof of Lemma~\ref{lemma:setS}), but it seems worth a little further thought concerning this. 
(See also the conjectures 
at the end of this section.)
Given a surface $S$, $n \in \N$ and a graph $G \in \cE^S_n$, let $\ext(G,S)$ be the number of graphs $G' \in \cE^S_{n+1}$ such that $G'$ restricted to $[n]$ is $G$; and let $\minext(n,S)$ be the minimum value of $\ext(G,S)$ over all $G \in \cE_n^S$.
Corollary 11 in~\cite{TherandomPlanar} shows (essentially) that
$\minext(n,\bS_0) \geqslant 6n-9$.  Given a sequence $S_n$ of surfaces, we can give a good estimate of the value $\minext(n,S_n)$ as long as the surface $S_n$ has Euler genus $o(n)$.
\begin{proposition} \label{prop.ext}
For each $n \geqslant 4$, $\minext(n,\bS_0)= 6n-9$; and if $g(n)=o(n)$ then $\, \minext(n,\bS_{\lfloor g(n)/2 \rfloor})= 6n+o(n)$ and $\, \minext(n,\bN_{g(n)})= 6n+o(n)$.
\end{proposition}
Note that for example $\minext(n,\bN_{g(n)})$ is defined to be the minimum over all graphs $G \in \cN\cE^{g}_n$ of $\ext(G,\bN_{g(n)})$, that is of the number of graphs $G' \in \cN\cE^{g(n)}_{n+1}$ (not $g(n\!+\!1)$ here) such that $G'$ restricted to $[n]$ is $G$.  For the proof of Proposition~\ref{prop.ext} we use two lemmas.
\begin{lemma} \label{lem.tri}
Let $h \geqslant 0$ and $n \geqslant 3$. Let $S$ be a surface of Euler genus $h$, and let the $n$-vertex graph $G$ have a cellular embedding in $S$ which is a triangulation with no non-contractible 3-cycles.  Then $\ext(G,S) = 6n+5h-9$, except if $h=0$ and $n=3$ when $\ext(G,S) = 8$ (not 9).
\end{lemma}
\begin{proof}
The embedding of $G$ in $S$ is unique, see Theorem 5.3.4 of~\cite{GraphsonSurfaces}.
In each graph $G'$ on $[n+1]$ embeddable in $S$ and such that $G'$ restricted to $[n]$ is $G$, the neighbours of vertex $n+1$ must form a subset of the vertices on a single face of the triangulation.  In the embedding of $G$ there are $e=3(n+h-2)$ edges and $f=2(n+h-2)$ faces.
Unless $h=0$ and $n=3$ the faces have distinct vertex sets (each of size 3), so
\[ \ext(G,S) = 1+n+e+f = 1+n+ 5(n+h-2) = 6n+5h-9.\]
If $h=0$ and $n=3$ then
\[ \ext(G,S) = 1+n+e+1 = 8,\]
which completes the proof.
\end{proof}

\begin{lemma} \label{lem.extub}
There exists $\delta>0$ such that, for all $n \geqslant 4$ and $0 \leqslant h \leqslant \delta n$,
\begin{equation} \label{eqn.extub}
 \minext(n,\bS_{\lfloor h/2 \rfloor}), \, \minext(n,\bN_{h}) \; \leqslant 6n+5h-9.
\end{equation}
\end{lemma}
\begin{proof}
There is a constant $c>0$ such that for all surfaces $S$ of Euler genus $h \geqslant 1$ there is an $n$-vertex (simple) triangulation of $S$ with $n \leqslant ch$, see Section 5.4 of \cite{GraphsonSurfaces}.
By subdividing each edge, inserting a vertex in each face and re-triangulating, we see that, with a larger constant $c'$,  we may insist that there are no non-contractible 3-cycles. Let $\delta=1/c'$.
Then for all $n \geqslant 3$ and all surfaces $S$ of Euler genus $h$ such that $0 \leqslant h \leqslant \delta n$ (including $h=0$) there is an $n$-vertex triangulation of $S$ with no non-contractible 3-cycles, and so~(\ref{eqn.extub}) follows from Lemma~\ref{lem.tri}.
\end{proof}

\begin{proof}[Proof of Proposition~\ref{prop.ext}]
Consider first the planar case.
Let $n \geqslant 4$ and let $G_0 \in \cP_n$. By adding edges if necessary we can form a graph $G' \in \cP_n$ which triangulates $\bS_0$; and $\ext(G_0,\bS_0) \geqslant \ext(G',\bS_0)$.
But by Lemma~\ref{lem.tri}, $\ext(G',\bS_0)= 6n-9$, and so
$\minext(n,\bS_0)= 6n-9$, as required.
\smallskip

Now consider the second part of the proposition. Let the graph $G_0$ on $[n]$ be embeddable in the surface~$S$ of Euler genus $h$. Add edges to $G_0$ if necessary to obtain an edge-maximal graph $G$ embeddable in $S$, and recall that $\ext(G_0,S) \geqslant \ext(G,S)$. Suppose that $G$ has $e$ edges and $f_3$ 3-faces.  Then $\ext(G,S) \geqslant 1+n+e+f_3$.
For in a graph $G'$ on $[n\!+\!1]$ with restriction to $[n]$ being $G$, vertex $n\!+\!1$ may be isolated, may be adjacent to any one vertex of $G$, may be adjacent to both ends of any one edge of $G$, or may be adjacent to all 3 vertices in any 3-face of $G$ (and the 3-faces must have distinct sets of incident vertices).

By~\cite{purity} there is an absolute constant $c$ such that by adding at most $ch$ edges to $G$ we may form a multigraph $G'$ which triangulates $S$.  By Euler's formula~(\ref{eqn.Euler}), $G'$ has $3(n+h-2)$ edges and $2(n+h-2)$ faces.  It follows that $G$ has at least $3(n+h-2)-ch$ edges and at least $2(n+h-2)-2ch$ 3-faces.  Hence
\[\ext(G,S) \geqslant 1+n+e+f_3 \geqslant 1+n+ 3(n+h-2)-ch + 2(n+h-2)-2ch = 6n+ (5-3c)h-9. \]
Thus $\minext(n,S) \geqslant 6n+O(h)$. But by Lemma~\ref{lem.extub}  we have the reverse inequality $\minext(n,S) \leqslant 6n+O(h)$, and we are done.
\end{proof}
\medskip

\noindent\emph{Better bounds?}

So far, we managed only to obtain lower bounds on the ratio $|\cE^S_{n+1}| / |\cE^S_{n}|$ 
(as $n$ increments by 1), with no upper bounds (if $S$ is not fixed). Using $\minext(n,S)$ does not give a tight lower bound on this ratio.
For every surface $S$, we know that $|\cE^S_n|/ n |\cE^S_{n-1}| \to \gamma_\cP$ as $n \to \infty$, and similarly for the connected graphs in $\cE^S$, see the asymptotic formulae~(\ref{eqn.limitprobG}) and (\ref{eqn.limitprobH}), 
and~\cite{bcr08,asymptLabGraphs,LimitLawsFixedS,Asymptoticformula}. 
The following conjecture is similar to 
\cite[Conjecture 117]{thesis}.
\begin{conjecture}
\label{conj.nincr1}
For any $\eps >0$ there is an $n_0$ such that for each $n\geqslant n_0$ and each surface $S$ 
\[\left|\cE^S_{n+1}\right| / \left|\cE^S_n\right| \geqslant (1-\eps) \, \gamma_{\cP}\, n \, \,, \]
and similarly for the connected graphs in $\cE^S$.
\end{conjecture}
\noindent
%
%
Being more precise (and more speculative), we may go further and ask whether, for each $n \in \N$ and each surface $S$ 
\begin{equation} \label{eqn.nadd1}
\left|\cE^S_{n+1}\right|/\left|\cE^S_n\right| \, \geqslant \, \left|\cP_{n+1}\right|/\left|\cP_n\right| \text{ ;}
\end{equation}
or even, if $S^+$ is obtained from $S$ by adding a handle or crosscap, then
\begin{equation} \label{eqn.nadd1b}
\left|\cE^{S^+}_{n+1}\right|/\left|\cE^{S^+}_n\right| \, \geqslant \left|\cE^S_{n+1}\right|/\left|\cE^S_n\right|\,.
\end{equation}
(Observe that~(\ref{eqn.nadd1}) would imply Conjecture~\ref{conj.nincr1},
and~(\ref{eqn.nadd1b}) would imply~(\ref{eqn.nadd1}).)
\medskip

Many of the results in the companion paper~\cite{MSrandom} depend on results in the present paper, but Theorem 3 of that paper 
does not.  By that result,
for all $0 \leqslant h \leqslant \tfrac1{12}n^2$, as $n \to \infty$ most graphs in $\cA_n^h$  have at least $n + h$ edges; and it follows easily that, for all $0 \leqslant h \leqslant \tfrac1{12}n^2$ we have
$\big|\cA_{n+1}^h\big|/\big|\cA_{n}^h \big| \geqslant (1+o(1))\, (2n + h)$, see the proof of inequality~(\ref{eqn.minext}). On the other hand, we noted above that
$\left|\cP_{n+1}\right|/\left|\cP_{n}\right| \sim \gamma_{\cP}\,n$ by equation~(\ref{incrementingn}). Hence (recalling that $\gamma_\cP<28$), we may see that 
the conjectured inequality~(\ref{eqn.nadd1}) holds
for all sufficiently large $n$ and all surfaces $S$ of Euler genus $h$ such that $26 n \leqslant h \leqslant \tfrac1{12}n^2$. 

\medskip 
 
The above discussion 
concerned finding better \emph{lower} bounds on the growth ratio as $n$ is incremented by~1, but it would be 
useful to find some \emph{upper} bounds.
We give one weak upper bound:
given a genus function $g(n)=o(n)$ there is an $n_0$ such that for all $n \geqslant n_0$ and $0 \leqslant h \leqslant g(n)$
\begin{equation} \label{eqn.ub2}
|\cA_{n+1}^{h}|/|\cA_n^h| \leqslant n^7.
\end{equation}
To prove~(\ref{eqn.ub2}), let $n_0 \geqslant 12$ be sufficiently large that $6(g(n)-2) \leqslant n$ for all $n \geqslant n_0$. Let $n \geqslant n_0$, let $0 \leqslant h \leqslant g(n)$, and let $G \in \cA^h_{n+1}$. 
By Euler's formula, $e(G) \leqslant 3 (n+1+h-2)$, so there is a vertex $v_0$ of degree at most $6+ \lfloor 6(h-2)/(n+1) \rfloor =6$. 
The graph $G^- = G -v_0$ is an $n$-vertex graph in $\cA^h$ with vertex set contained in $[n+1]$, and the number of such graphs is $(n+1) \, |\cA^h_n|$.  Let $d= \sum_{j=0}^{6} \binom{n}{j} \leqslant \tfrac12 n^{6}$.
Each graph $G^-$ is constructed at most $d$ times, since to reconstruct $G$ we need just to guess the at most $6$ neighbours of the `missing' vertex $v_0$. Hence
\[  |\cA^h_{n+1}| \leqslant (n+1)\,| \cA^h_n|\, d \leqslant | \cA^h_n| \, n^7, \]
giving~(\ref{eqn.ub2}). 
There is no result like this if $h$ is not bounded by a suitable function of $n$\,: for an extreme example, if
say $h \geqslant \tfrac16 n^2$ then
$K_{n+1} \in \cA^h$ 
and so
$|\cA_{n+1}^h|/|\cA_n^h| = 2^n$.
However, surely we can improve on the upper bound~(\ref{eqn.ub2})?
\begin{conjecture} \label{conj.grn}
There is a constant $\alpha$ such that, for all $n \geqslant 1$ and $0 \leqslant h \leqslant n$, the growth ratio $|\cA_{n+1}^h|/|\cA_n^h|$
is at most $\alpha n$.
\end{conjecture}
\noindent
In this conjecture we would hope to be able to take $\alpha$ close to the planar graph growth constant $\gamma_{\cP}$.


\medskip

\subsection{Growth ratios when adding a handle}
\label{gr_handle}

We now consider the growth ratio of the graph classes when we increment the genus bound $h$ by 2. When $h$ is fixed, by~(\ref{eqn.limitprobG}) the growth ratio $|\cO\cE_{n}^{h+2}|/|\cO\cE_n^h|$ (as $h$ is incremented by 2) is asymptotic to $n^{5/2}$ as $n \to \infty$, and by~(\ref{eqn.limitprobH}) the growth ratio $|\cN\cE_{n}^{h+1}|/|\cN\cE_n^h|$ (as $h$ is incremented by 1) is asymptotic to $n^{5/4}$ as $n \to \infty$.
\begin{lemma} \label{lem.add1tog}
For every $h \geqslant 0$ and $n \geqslant 1$
\begin{equation}\label{eqn.lbGH}
\left|\cA_n^{h+2}\right| \geqslant \frac{\binom{n}{2}-3(n+h)}{3(n+h)}\, \left|\cA_n^{h}\right| \text{.}
\end{equation}
\end{lemma}
\begin{proof} 
We prove these inequalities by a simple double counting argument: for each graph $G$ in $\cA_n^h$ we show that we can construct many graphs $G'$ in $\cA_n^{h+2}$, and each graph $G'$ is not constructed too many times.
We make frequent use of such double-counting arguments.

Given a surface $S$ and a graph $G$ embedded in $S$, by adding a handle or twisted handle to the surface we can add any one of the non-edges to form a new graph $G'$. (We can attach the handle to the surface inside two faces incident to the two vertices we wish to connect, and then add the new edge along the handle; and similarly for a twisted handle.)  The only time we need the handle to be twisted is when $h=0$ (so the surface $S$ is the sphere $\bS_0$) and we need $G'$ to be embeddable in a non-orientable surface.  Thus if $G \in \cA^h$ then $G' \in \cA^{h+2}$.

Each graph in $\cA_n^h$ has at most $3(n+h-2) \leqslant 3(n+h)$ edges. This means that from each graph $G\in \cA_n^{h}$ we construct at least $\binom{n}{2}-3(n+h)$ graphs $G' \in \cA_n^{h+2}$. Furthermore, each graph $G'$ constructed has at most $3(n+h-2)+1 \leqslant 3(n+h)$ edges, and so is constructed at most this many times. The inequality~(\ref{eqn.lbGH}) follows.
\end{proof}

Let $n \geqslant 1$ and $0 \leqslant h \leqslant \tfrac1{43} n^2$. Then
\[ \tfrac13\left(\binom{n}{2}-3(n+h)\right) = \tfrac16 \left(n^2-7n-6h\right) \geqslant \tfrac16 \left( \tfrac{37}{43}n^2 -7n\right) \geqslant \tfrac17 n^2  \]
for $n$ sufficiently large, since $\tfrac{37}{43} > \tfrac67$.  Hence, by Lemma~\ref{lem.add1tog}, if $g(n) \ll n^2$ and $n$ is sufficiently large then
\begin{equation} \label{eqn.lbasymp}
\left|\cA_n^{g+2}\right| \geqslant \frac{n^2}{7(n+g)}\, \left|\cA_n^{g}\right| \,.
\end{equation}
Using the same argument as in the proof of Lemma~\ref{lem.add1tog} for $\cF$ we similarly obtain that, if $g(n) \ll n^2$ and $n$ is sufficiently large then
\begin{equation} \label{eqn.lbasympF}
\left|\cF_n^{g+2}\right| \geqslant \frac{n^2}{7(n+g)}\, \left|\cF_n^{g}\right| \,.
\end{equation}
We shall use this inequality in the proof of Theorem~\ref{thm.lowerboundab}.

We noted at the start of this subsection that when $h$ is fixed, $|\cO\cE_{n}^{h+2}|/|\cO\cE_n^h| \sim n^{5/2}$ and $|\cN\cE_{n}^{h+1}|/|\cN\cE_n^h| \sim n^{5/4}$ 
as $n \to \infty$. (When $h$ is not fixed we do not have a useful lower bound on $|\cN\cE_n^{h+1}|/|\cN\cE_n^h|$.) As in the `adding a vertex' case, we have no useful upper bounds on the growth ratios $|\cA_{n}^{h+2}|/|\cA_n^h|$ as $h$ is incremented by 2 (with $n$ fixed).
Conjecture~\ref{conj.gcP} in Section~\ref{sec.concl} (concerning the growth constant $\gamma_\cP$) would be implied by the following conjecture -- in which perhaps we could take $\beta=2$?
\begin{conjecture} \label{conj.grh}
There are constants $\alpha, \beta$ such that $|\cA_{n}^{h+2}|/|\cA_n^h| \leqslant \alpha\, n^{\beta}$ for all $0 \leqslant h \leqslant n$.
\end{conjecture}


\section{Lower bounds on $|\cA^g_n|$, proofs of Theorem~\ref{thm.lowerboundab} and~\ref{thm.lowerbound}}
\label{sec.lb}

In this section we prove 
the two parts (a) and (b) of Theorem~\ref{thm.lowerboundab} (which give lower bounds on $|\cF^g|$ and thus on $| \cA^g|$),
then quickly prove Corollary~\ref{cor.lb}, and finally prove the two parts (a) and (b) of Theorem~\ref{thm.lowerbound}. 

\subsection{Proof of Theorem~\ref{thm.lowerboundab} (a)}
\begin{proof}[Proof of Theorem~\ref{thm.lowerboundab} (a)]
By inequality~(\ref{eqn.lbasympF}) 
there is an $n_0$ such that if $n \geqslant n_0$ and $0 \leqslant h \leqslant \tfrac17 n$ then
\[ \left|\cF_n^{h+2}\right| \geqslant \frac{n^2}{7(n+h)}\cdot \left|\cF_n^{h}\right| \geqslant \frac{n}{8}\cdot \left|\cA_n^{h}\right| .\]
Applying this $\lfloor g(n)/2\rfloor$ times starting with $\cP_n$ we see that, if $n \geqslant n_0$ and $g(n) \leqslant \tfrac17 n$, then
\[    \left|\cF_n^{g}\right|\geqslant
    \left|\cP_n\right| \, \left(n/8\right)^{\lfloor g/2\rfloor} \text{ .} \]
But by~(\ref{eqn.limitprobG}) 
\[ \left|\mathcal{P}_n\right| \sim c^{(0)} n^{-7/2}\, \gamma_{\cP}^n \, n!\,, \]
so when $g(n)$ is $o(n)$
\[ \left|\cF_n^{g}\right| \geqslant (1+o(1)) \, c^{(0)} n^{-7/2} \: \gamma_{\cP}^n \, n! \left(n/8\right)^{ (g-1)/2}  
 = (1+o(1))^n \, \gamma_{\cP}^n \, n!\,  n^{g/2}\,, \]
 and Theorem~\ref{thm.lowerboundab} (a) follows.
\end{proof}

\subsection{Proof of Theorem~\ref{thm.lowerboundab} (b)}
\label{subsec.proof2b}


For integers 
$h \geqslant 0$ let $\cC^h$ be the class of connected graphs in $\cF^h$, that is, the class of connected graphs $G$ 
with at most $h+v(G)-1$ edges.  Thus for example $\cC^0$ is the class of trees.
Of course $\cC^h \subseteq \cF^h$.

\begin{lemma}\label{lem.lb}
For all $n \geqslant 1$ and $0 \leqslant h \leqslant \tfrac12 n^2 - \tfrac52 n$ we have
\begin{equation} \label{eqn.lb-allg} \notag
     \left|\cC_n^{h}\right| \geqslant n^{n-2}\cdot \left(\frac{n^2-3n}{2(n+h)}\right)^{h}\text{ .}
\end{equation}
\end{lemma}

\begin{proof}
How many connected graphs are there on $[n]$ with exactly $h+n-1$ edges?  Pick a spanning tree on $n$ vertices, there are $n^{n-2}$ such trees; and then add any $h$ of the $\binom{n}{2}-n+1$ potential edges that are not yet present in the graph, there are $\binom{\binom{n}{2}-n+1}{h}$ choices for this. Thus in total there are $n^{n-2} \cdot \binom{\binom{n}{2}-n+1}{h}$ constructions of connected graphs, each with $h+n-1$ edges.
Also, each graph is constructed at most $\binom{h+n-1}{h}$ times, since there are at most this number of choices for the $h$ added edges.
Further, the conditions on $n$ and $h$ imply that $\frac{1}{2}(n^2-3n) \geqslant n+h$. Thus 
\[ |\cC^h_n| \; \geqslant \; n^{n-2} \cdot \frac{\left(\frac{1}{2}(n^2-3n+2)\right)_{(h)}}{(n+h-1)_{(h)}} \; \geqslant \; n^{n-2}\cdot \left(\frac{n^2-3n}{2(n+h)}\right)^{h}\
\]
as required.
\end{proof}

\begin{proof}[Proof of Theorem~\ref{thm.lowerboundab} (b)]
Let us check first that
\begin{equation} \label{eqn.nfact}
n! \leqslant n^{n+1}e^{-n} \;\; \mbox{ for all } n \geqslant 7
\end{equation}
by induction on $n$ (see also for example (4) in Chapter 1.1 of~\cite{BelaB}, or Exercise 24 of~\cite{ArtOfComp} for a similar inequality).
Direct computation shows that the inequality~(\ref{eqn.nfact}) holds for $n=7$. Suppose that it holds for some integer $n \geqslant 7$.  Then
\begin{eqnarray*}
(n+1)! & = & (n+1)\, n! \;\; \leqslant \;\; (n+1)\, n^{n+1}e^{-n}\\
& = &
(n+1)^{n+2}\, (1- \tfrac{1}{n+1})^{n+1}\, e^{-n}\;\; \leqslant \;\;
(n+1)^{n+2}\, e^{-(n+1)}
\end{eqnarray*}
since $1-x \leqslant e^{-x}$ for all $x$ and so
$(1- \tfrac{1}{n+1})^{n+1} \leqslant e^{-1}$.  This completes the proof of~(\ref{eqn.nfact}).

We shall consider $h \leqslant \tfrac13 n^2$. If $n \geqslant 15$ then
\[ (\tfrac12 n^2 - \tfrac52 n) - \tfrac13 n^2 = \tfrac16  n(n-15) \geqslant 0;\]
so the conditions in Lemma~\ref{lem.lb} hold when $n \geqslant 15$ and $0 \leqslant h \leqslant \tfrac13 n^2$.
Also $\tfrac12(n^2 -3n) \geqslant \tfrac13 n^2$ when $n \geqslant 9$. Thus by Lemma~\ref{lem.lb}, for all $n \geqslant 15$ and $0 \leqslant h \leqslant \tfrac13 n^2$, 
\begin{equation}
\label{firstlowerorientablebound1}
    \begin{split}
        \left|\cC_n^{h}\right| &\geqslant n^{n-2} \, \left(\frac{n^2-3n}{2(n+h)}\right)^{h}\notag
        \; \geqslant \; n^{n-2} \, \left(\frac{n^2}{3\,h(1+n/h)}\right)^{h} \notag\\
        &= n^{n-2} \, \left(\frac{n^2}{3h}\right)^{h} \, \left(1+\frac{n}{h}\right)^{-h} \notag\\
        &\geqslant n^{n-2} \, \left(\frac{n^2}{3h}\right)^{h} \, e^{-n}\notag
        \; \geqslant \; n^{-3}  \, \left(\frac{n^2}{3h}\right)^{h} \, n! 
    \end{split}
\end{equation}
where in the last step we use the inequality~(\ref{eqn.nfact}).
Thus, for all $n \geqslant 15$ and $h \geqslant 0$
\begin{equation} \label{eqn.Alb} 
\left|\cC_n^{h}\right| \geqslant n^{-3} \, \left(\frac{n^2}{3h}\right)^{h}\, n! \, 
\end{equation}
where the inequality holds for $h > \tfrac13 n^2$ since $|\cC^h_n| \geqslant |\cC^0_n| = n^{n-2} \geqslant n^{-3} n!$\,.
Theorem~\ref{thm.lowerboundab} (b) now follows (since $\cC^h \subseteq \cF^h$).
\end{proof}

\subsection{Proof of Corollary~\ref{cor.lb}}
Since $\cA^h_n \supseteq \cF^h_n$ we need only to consider $\cF^h_n$.  Let $c_0 \geqslant 1$.
Let $c_1$ be the constant in Theorem~\ref{thm.lowerboundab} (b) : 
then for all $0 \leqslant h \leqslant c_0n$
\[ 
|\cF^h_n|/n! \geqslant c_1^{n+h} (n^2/h)^h \geqslant c_1^{n+h} (h/c_0^2)^h = c_1^n\, (c_1/c_0^2)^h\, h^h.\]
If $c_1/{c_0}^2 \geqslant 1$ then $
|\cF^h_n|/n! \geqslant c_1^n h^h$ for all $0 \leqslant h \leqslant c_0n$, so we may set $c=c_1$.
On the other hand, if $c_1/{c_0}^2 < 1$, then for all $0 \leqslant h \leqslant c_0n$
\[ 
|\cF^h_n|/n! \geqslant c_1^n (c_1/c_0^2)^{c_0 n}\, h^h = (c_1\, (c_1/c_0^2)^{c_0})^n\, h^h \,,\]
so we may set $c= c_1 (c_1/c_0^2)^{c_0}$.
We have now shown that $|\cF^h_n|/n! \geqslant c^n h^h$ for all $0 \leqslant h \leqslant c_0n$.
Finally we have
\[\left|\tA_n^h\right|\geqslant 
\left|\widetilde{\cF}_n^h \right|
\geqslant \left| \cF^h_n \right|/n! \geqslant c^n \, h^h\,, \]
and the proof is complete.

\subsection{Proofs of Theorem~\ref{thm.lowerbound} (a) and (b)}
In the proofs here we use results that give upper bounds on the Euler genus of \emph{most} graphs with a given number of edges \cite{GenusRandom1, GenusRandom3, GenusRandom2}. 
Some of these results are stated for the Erd\H{o}s-R\'enyi random graph $G(n,p)$ with given edge probability $p=p(n)$, but they can easily be applied 
to the case of a given number $m=m(n)$ of edges, as pointed out in \cite{GenusRandom1, GenusRandom3}. 

\begin{proof}[Proof of Theorem~\ref{thm.lowerbound} (a)]
Let $\eps > 0$. Let $m=m(n)= \lfloor (1- \tfrac12 \eps)\, \tfrac{j+2}{j}\, (g(n)-1) \rfloor$. Then $n^{1+1/(j+1)} \ll m \ll n^{1+1/j}$, and so
$n^{-j/(j+1)} \ll m/\binom{n}{2} \ll n^{-(j-1)/j}$.
It follows from 
\cite{GenusRandom2} (see (1.2) in~\cite{GenusRandom2} for the $G(n,p)$ version) that almost every graph on $n$ vertices with $m$ edges can be embedded in an orientable surface of Euler genus at most
\[ (1+ \tfrac12 \eps) \tfrac{j}{j+2} \tfrac{m}{\binom{n}{2}} 
\, n^2 \leqslant (1- \tfrac14 \eps^2)\, \tfrac{n}{n-1} (g(n)-1) \leqslant g(n)-1\]
for $n$ sufficiently large.
From Observation \ref{nonor_from_or} it then follows that almost every graph on $n$ vertices with $m$ edges can be embedded in a non-orientable surface of Euler genus at most $g(n)$. So, for $n$ sufficiently large, at least half of the graphs on $n$ vertices with $m$ edges lie in the class $\cA_n^g$. 
Hence
\begin{equation}
\begin{split}
    |\cA_n^g| &\geqslant \tfrac{1}{2} \binom{\binom{n}{2}}{m}
    \; \geqslant \; \tfrac{1}{2} \left(\frac{n(n-1)}{2m}\right)^{m}\\
    &\geqslant c^{g(n)}(n^2/g(n))^{(1- \tfrac12 \eps)\frac{j+2}{j} g(n)} \;\;\; \mbox{ for some constant } c>0\\
    &\geqslant (n^2/g(n))^{(1-\eps)\frac{j+2}{j} g(n)} \;\;\;\; \mbox{ for $n$ sufficiently large}\,,
\end{split} \nonumber
\end{equation}
as required.
\end{proof}

\begin{proof}[Proof of Theorem~\ref{thm.lowerbound} (b)]
We first prove equation~(\ref{eqn.gbig1}).
Let $g(n) \gg n^{3/2}$ and let $\bar{g}(n) = 
\min\{g(n), \lfloor \tfrac1{12} n^2 \rfloor\}$. 
Denote $\binom{\tbinom{n}{2}}{j}$ by $x(n,j)$ for each integer $j \geqslant 0$.
To prove the lower bound in~(\ref{eqn.gbig1}) we consider two overlapping cases. 

Assume first that $n^{3/2} \ll g(n) \ll n^2$ (so $\bar{g}(n)=g(n)$ for $n$ sufficiently large).  Then by the case $j=1$ of part~(c) we have $|\cA^g_n| \geqslant (n^2 /g)^{(1+o(1)) 3g}$. But
\begin{equation} \label{eqn.lb0}
x(n,3g) \leqslant \left(\frac{en^2}{6g}\right)^{3g} = \left(\frac{n^2}{g}\right)^{(1+o(1)) 3g} \end{equation}
so $|\cA^g_n| \geqslant x(n,3\bar{g})^{(1+o(1))}$,
which is the required lower bound.

Now assume that $g(n) \gg (\log n)^2 \, n^{3/2}$.
Let $0< \eps<1$ and let $m=m(n)  \sim (1-\eps)\, 3 \,\bar{g}(n)$.
Then $p=m/\binom{n}{2}$ satisfies $p^2(1-p^2) \gg (\log n)^4/n$; and hence it follows from \cite[Theorem 4.5]{GenusRandom1} that almost every graph on $n$ vertices with $m$ edges embeds in an orientable and a non-orientable surface of Euler genus at most $(1+\eps)\,\tfrac13 m \leqslant g(n)$.
Thus
\begin{equation} \label{eqn.lbm}
 |\cA_n^g| \geqslant (1+o(1))\, x(n,m) \,.
\end{equation}
%
Let $m_1= 3(n+g(n)-2)$, let $m_2= \min\{m_1, \tfrac12 \binom{n}{2}\}$, let $m_3 = 3(n+\bar{g}(n))$,
and note that $m_2 \leqslant m_3$.
Since every graph in $\cA_n^g$ has at most $m_1$ edges,
\begin{equation} \label{eqn.ubm1}
|\cA_n^g| \leqslant \sum_{j \leqslant m_1} x(n,j) \leqslant 2 \sum_{j \leqslant m_2} x(n,j) \leqslant 2 \sum_{j \leqslant m_3} x(n,j) \,.
\end{equation}
The numbers $x(n,j)$ 
are increasing for $j=0,1,\ldots, 3\bar{g}$ (since $3 \bar{g} \leqslant \frac12 \binom{n}{2}$).
Also, for each $0 \leqslant j \leqslant \binom{n}2$
\begin{equation} \label{eqn.binom}
\frac{x(n,j+1)}{x(n,j)} = \frac{\binom{n}{2}- j}{j+1} \leqslant \frac{n(n-1)}{2j} \;\; \mbox{ and } \;\; x(n,j) \geqslant \left(\frac{n(n-1)}{2j}\right)^j\,.
\end{equation}
Thus
\[ \frac{x(n,3\bar{g})}{x(n,m)} \leqslant \left(\frac{n(n-1)}{2m}\right)^{3\bar{g}-m} \leqslant x(n,m)^{\tfrac{3\bar{g}-m}{m}} = x(n,m)^{\tfrac{\eps}{1-\eps} +o(1)}\]
and so
\[ x(n,3\bar{g}) \leqslant x(n,m)^{1 + \tfrac{\eps}{1-\eps}+o(1)}\,.\]
Hence by~(\ref{eqn.lbm}) we have
\begin{equation} \label{eqn.lb2} |\cA_n^g| \geqslant x(n,3\bar{g})^{1+o(1)}\,.
\end{equation}
This completes the proof of the lower bound in~(\ref{eqn.gbig1}).

Now we prove the upper bound in~(\ref{eqn.gbig1}).
Consider the numbers $x(n,j)$ for $j=3\bar{g}+1,\ldots,m_3$.  For each such~$j$, by~(\ref{eqn.binom})
\[ \frac{x(n,j)}{x(n,3\bar{g})} \leqslant \left(\frac{n(n-1)}{6\bar{g}}\right)^{j-3\bar{g}} \leqslant \left(\frac{n(n-1)}{6\bar{g}}\right)^{3n} \leqslant x(n, 3\bar{g})^{n/\bar{g}}\,.\]
Thus
\[ \sum_{j=3\bar{g}+1}^{m_3} x(n,j) \leqslant 3n \cdot x(n,3\bar{g})^{1+ n/\bar{g}} = x(n,3\bar{g})^{1+O(n/\bar{g})}\,. \]
Also by monotonicity
\[ \sum_{j=0}^{3 \bar{g}} x(n,j) \leq
(3 \bar{g}+1) \cdot x(n,3 \bar{g}) = x(n,3\bar{g})^{1+o(1)}\,.\]
Hence by the inequality~(\ref{eqn.ubm1})
\[ |\cA^g_n| \leqslant 2 \sum_{j=0}^{m_3} x(n,j) \leqslant x(n,3\bar{g})^{1+o(1)}\,,\]
and we have proved the upper bound in~(\ref{eqn.gbig1}).  This completes the proof of equation~(\ref{eqn.gbig1}), namely that
$|\cA^g_n| = x(n,3\bar{g})^{1+o(1)}$.

Finally we deduce equations~(\ref{eqn.gbig2}) and~(\ref{eqn.gbig3}) from equation~(\ref{eqn.gbig1}).
Note first that if $g(n) \leqslant \frac13 \binom{n}2$ then 
as in~(\ref{eqn.lb0}) and~(\ref{eqn.binom})
\[ \left(\frac{n(n-1)}{6 g}\right)^{3 g} \leqslant x(n,3 g) \leqslant   \left(\frac{e\,n^2}{6 g}\right)^{3 g}\,,\]
and so if $g(n) \ll n^2$ then $x(n,3 \bar{g}) = (\tfrac{n^2}{g})^{(1+o(1))\,3g}$.
This gives equation~(\ref{eqn.gbig2}).
Now suppose that $g(n) \sim c n^2$ for some $0<c \leqslant \tfrac1{12}$.
Then $3g \sim 6c\, \binom{n}2$, so
$x(n,3 \bar{g})= x(n, 3g) = 2^{(1+o(1)) H(6c) \tbinom{n}2}$ (see for example \cite[Example 11.1.3]{inforTheory}),
and equation~(\ref{eqn.gbig3}) follows.
\end{proof}

\section{Upper bounds on $|\cA^g_n|$, proof of Theorem~\ref{thm.upperbound}}
\label{sec.ub}
In this section we shall prove an upper bound on numbers of  maps (Theorem~\ref{thm.ubmap}), from which we shall deduce Theorem~\ref{thm.upperbound}.
We call a cellularly embedded connected pseudograph, considered as an unlabelled object, a \emph{map} (where in general we do not specify a root).
We also deduce Corollary~\ref{corollary_gg} in Section~\ref{proof_corollary_gg}.

\begin{theorem}\label{thm.ubmap}
There are constants $c$ and $n_0$ such that, for all $n \geqslant n_0$ and all $h \geqslant 0$, the number of $n$-vertex simple maps in a surface of Euler genus $h$ is at most $c^{n+h} \: h^{h}$.
\end{theorem}
The proof will show that we may take $c= 2.3 \times 10^5 $; and if we consider only orientable surfaces, we may take $c=624$. 
This result will quickly give Theorem~\ref{thm.upperbound}, using one preliminary lemma. A set $\cA$ of graphs is called \emph{bridge-addable} when for each graph $G$ in $\cA$, if $u$ and $v$ are vertices in distinct components of $G$ then the graph obtained from $G$ by adding an edge between $u$ and $v$ is also in $\cA$.
\begin{lemma} [\cite{b-add-unlab-conn}] \label{lem.uconn}
Let $\cA$ be a bridge-addable set 
of graphs and let $\cC$ be the set of connected graphs in~$\cA$.  Then 
\begin{equation} 
|\widetilde{\cC}_n | \geqslant  |\widetilde{\cA}_n | \, / \, 2n \;\; \mbox{ for each } n \in \N \, .\notag
\end{equation}
\end{lemma}
\noindent
(Stronger results are known for labelled graphs and conjectured to hold for unlabelled graphs, see~\cite{b-add-unlab-conn}.)

\begin{proof}[Proof of Theorem~\ref{thm.upperbound} (using Theorem~\ref{thm.ubmap})]
Let $c \geqslant 2$ and $n_0$ be as in Theorem~\ref{thm.ubmap}, and let $n \geqslant n_0$. The number of connected unlabelled $n$-vertex graphs embeddable in a surface of Euler genus at most $h$ is at most the total number of $n$-vertex simple maps in a surface of Euler genus $k$ for $0 \leqslant k \leqslant h$, see Section~\ref{subsec.embed}.
By Theorem~\ref{thm.ubmap} this total is at most \[\sum_{k=0}^h c^{n+k} k^k \leqslant c^n h^h \sum_{k=0}^h c^k \leqslant 2\, c^{n+h} h^h\,.\]
But the set of $n$-vertex graphs embeddable in a surface of Euler genus at most $h$ is bridge-addable, so by Lemma~\ref{lem.uconn}
the number of unlabelled $n$-vertex graphs embeddable in a surface of Euler genus at most $h$ is at most $2n$ times the corresponding number of connected graphs, so $|\tE_n^h| \leqslant 4n \, c^{n+h} h^h$,
which yields Theorem~\ref{thm.upperbound}.
\end{proof}

To prove Theorem~\ref{thm.ubmap} we shall first show how to upper bound numbers of maps by numbers of unicellular maps.  In the orientable case, there is a formula for the number of  unicellular maps rooted at an oriented edge, and we can complete the proof quickly, in Section~\ref{subsec.62}.  In the non-orientable case, we know only formulae (depending on parity) for numbers of `precubic' unicellular maps (with each vertex degree either 1 or 3, and a vertex of degree 1 specified as root), so we have to work harder, in Section~\ref{subsec.63}.  The upper bound for the orientable case follows from that for the non-orientable case (using Observation~\ref{nonor_from_or}), but it is useful to prove the bound for the orientable case as an introduction to the other harder case (and 
we can give a better value for the constant $c$.)


\subsection{From general maps to unicellular maps}

Given a map $M$ on a surface and a face $F$ of $M$, a \emph{chord} of $F$ in $M$ is a line between two vertices on the boundary of $F$ which apart from its two end points is embedded in the interior of $F$.
If a map has more than one face then it has an edge which is in two distinct facial walks.
Let us spell out how, when we start with a map which may have internally disjoint chords, and an edge which is in two distinct facial walks, we can move the edge from being part of the map to being a new chord.

Let the connected graph $G$ and the graph $H$ have the same vertex set and disjoint edge sets.  Let $G$ be cellularly embedded in a surface $S$, forming the map $M$, with the edges of $H$ (if any) embedded as internally disjoint chords of $M$.
Let the edge $e=uv$ be in two distinct facial walks of $M$, namely $F_1$ oriented to follow $uv$ and $F_2$ oriented to follow $vu$. (We use the same name for a face and the corresponding facial walk.) Let $F'_1$ be the $v-u$ walk obtained from $F_1$ by removing $uv$, similarly let $F'_2$ be the $u-v$ walk obtained from $F_2$ by removing $vu$, and let $F$ be the closed walk obtained by following $F'_1$ then $F'_2$. If we delete the edge $e$ from $G$ to form $G \backslash e$ and add $e$ to $H$ to form $H + e$, then $G \backslash e$ is connected, deleting $e$ from $M$ gives the map $M \backslash e$ in the same surface $S$, and $M \backslash e$ has the same faces as $M$ except that $F_1$ and $F_2$ are replaced by $F$ (and thus $M \backslash e$ has one less face than $M$).  Also, the edges of $H + e$ are embedded as internally disjoint chords of $M \backslash e$, with $e$ and any chords of $F_1$ or $F_2$ in $M$ embedded as chords of the new face $F$ of $M \backslash e$.
Applying this procedure repeatedly gives the following lemma.

\begin{lemma} \label{lem.chord-embed}
Let the connected graph $G$ be cellularly embedded in a surface $S$, forming the simple map $M$, and assume that $M$ has $f \geqslant 2$ faces. Then there is a set $X$ of $f-1$ edges of $G$ such that $G \backslash X$ is connected, $M \backslash X$ is a simple unicellular map in the original surface $S$, and the edges in $X$ are embedded as internally disjoint chords in the unique face of $M \backslash X$.
\end{lemma}

Let $Map(n,e,S)$ be the set of $n$-vertex $e$-edge simple maps in the surface $S$ (considered up to isomorphism). Similarly, let $Map(n,S)$ be the set of $n$-vertex simple maps in $S$, and let $U Map(n,S)$ be the set of $n$-vertex simple unicellular maps in $S$.
For $0 \leqslant j \leqslant k-3$ let $D(k,j)$ be the set of dissections of a $k$-gon on vertex set $[k]$ with $k+j$ edges.
If $M$ is a map and $F$ is a facial walk in $M$ of length~$t$, the corresponding \emph{polygon} is the simple convex polygon $P$ in the plane  obtained by creating a separate copy of a vertex $v$ for each visit of the walk to $v$ (and similarly a second copy of an edge if it is used twice), so $P$ has $t$ vertices and $t$ edges.  Internally disjoint chords of the face $F$ in $M$ form a dissection of the polygon $P$.

In Lemma~\ref{lem.chord-embed}, if $S$ has Euler genus $h$, and $M$ has $n$ vertices and $e$ edges, then by Euler's formula we have $f-1= e-n-h+1$ and the unicellular map has $n+h-1$ edges. Thus Lemma~\ref{lem.chord-embed} yields the next lemma.

\begin{lemma} \label{lem.mapstounimaps}
For each $h \geqslant 0$ and surface $S$ of Euler genus $h$, and each $n,e \in \N$ \[ |Map(n,e,S)| \leqslant | UMap(n, S)| \cdot | D(2(n\!+\!h\!-\!1), e\!-\!n\!-\!h\!+\!1)|.\]
\end{lemma}
By~\cite{enumeration_unlabeled_outerplanar}, for all sufficiently large~$k$, there are at most $(2+3\sqrt{2})^k$ dissections of a polygon with vertex set~$[k]$.
Thus from Lemma~\ref{lem.mapstounimaps} we obtain the following bound on numbers of maps in terms of numbers of unicellular maps. 
\begin{lemma} \label{lem.unimap-to-map}
For $n$ sufficiently large, for each $h \geqslant 0$ and surface $S$ of Euler genus $h$, 
\[ |Map(n,S)| \leqslant | UMap(n, S)| \cdot (2+3\sqrt{2})^{2n+ 2h-2}.\]
\end{lemma}

\subsection{Orientable case: unicellular maps and 
proof of Theorem~\ref{thm.ubmap}}
\label{subsec.62}
In this section we complete the proof of the orientable case of Theorem~\ref{thm.ubmap}.
We need just one more lemma.
\begin{lemma} \label{lem.unimap-or}
For $n \geqslant 1$ and even $h \geqslant 0$, the number $\tilde{f}_1(n,h)$ of unlabelled unicellular $n$-vertex maps in the orientable surface $\mathbf{S}_{h/2}$ is at most $\; 2^{4n+3h} \: h^{h}$.
\end{lemma}
\begin{proof}
Let $\tilde{f}_1^{(r)}(n,h)$ be the number of unlabelled rooted unicellular $n$-vertex maps in the orientable surface $\mathbf{S}_{h/2}$, where the root is an oriented edge.  By \cite{onefaceorientable} we have the exact formula 
\begin{equation}\label{unicellularrootedmapsa} \notag
\tilde{f}_1^{(r)}(n,h)
= \frac{(2n+2h-2)!}{2^{h} \, n! \, (n+h-1)!} \; \sum_{\substack{i_1+\cdots+i_n=h\\ i_1,...,i_n\geqslant 0}} \, \prod_{j=1}^{n}\frac{1}{2i_j+1}\;.
\end{equation}
The sum in the above equation is at most
\[\sum_{\substack{i_1+\cdots+i_n=h\\ i_1,...,i_n\geqslant 0}} 1= \binom{n-1+h}{n-1} \leqslant \binom{n+h}{n}\; ,\]
and of course
$\tilde{f}_1(n,h) \leqslant \tilde{f}_1^{(r)}(n,h)$.
Hence
\begin{eqnarray*} 
\tilde{f}_1(n,h)
& \leqslant &
\frac{(2n+2h-2)!}{2^{h} \, n! \, (n+h-1)!} \cdot \binom{n+h}{n}\\
& = &
2^{-h} \,  \binom{2n+2h-2}{n+h-1} \, \frac{(n+h-1)!}{n!} \, \binom{n+h}{n}\\
& \leqslant &
2^{-h} \, 2^{2n+2h-2} \,  \binom{n+h}{n}^2 \, h!\\
& \leqslant &
2^{2n+h} \, 2^{2n + 2h} \, h!
\;\; \leqslant \;\; 2^{4n+3h} \, h^h
\end{eqnarray*}
as required.
\end{proof}

We may now complete the proof of the orientable case of Theorem~\ref{thm.ubmap}.
By Lemmas~\ref{lem.unimap-to-map} and~\ref{lem.unimap-or}, there is an $n_0$ such that, for all $n \geqslant n_0$ and even $h \geqslant 0$,
the number of $n$-vertex simple maps in $\bS_{h/2}$ is at most
\[(2+3\sqrt{2})^{2n+ 2h-2} \, \tilde{f}_1(n,h) \leqslant  c_0^{n+h} h^h, \]
where $c_0= 2^4 \, (2+3\sqrt{2})^2 \approx 623.5$.

\subsection{Non-orientable case: unicellular maps and 
proof of Theorem~\ref{thm.ubmap}}
\label{subsec.63}

In the orientable case, in the proof of Lemma~\ref{lem.unimap-or} we started from a formula for the number $\tilde{f}_1^{(r)}(n,h)$ of $n$-vertex edge-rooted unicellular maps in $\bS_{h/2}$.  We have to work harder to complete the proof of Theorem~\ref{thm.ubmap} for non-orientable surfaces.  For convenience we first consider even values of the Euler genus~$h$: there is a formula for odd $h$ like that used in the proof of inequality~(\ref{claim.1}) for even $h$, but we do not need to use it.

Following~\cite{nonorientable1}, we say that a map is \emph{precubic} if each vertex degree is 1 or 3, and the map is rooted at a vertex of degree 1. 
For integers $n \geqslant 1$ and $h \geqslant 0$, we make the following definitions.  Recall that $UMap(n,\bN_h)$ is the set of $n$-vertex unicellular  maps in $\bN_h$ (where these maps are not rooted and not necessarily simple). Let $UMap(n,\bN_h,\ell)$ be the set of maps in $UMap(n,\bN_h)$ with exactly $\ell$ vertices of degree~2.
Let $PUMap(m,\bN_h)$ be the set of $m$-edge unicellular precubic maps in $\bN_h$.
Finally, let $PUMap\,(\leqslant m,\bN_h)$ be the set of unicellular precubic maps in $\bN_h$ with at most $m$ edges.
Lemma~\ref{lem.unimap-nonor} gives an upper bound on $|UMap(n,\bN_h)|$ like that in Lemma~\ref{lem.unimap-or} for the orientable case.

\begin{lemma} \label{lem.unimap-nonor}
For each $n \geqslant 1$ and even $h \geqslant 0$,
\[ |UMap(n,\bN_h)| \leqslant c^{n+h}\, h^h \]
where
$c= 2^7 e^{3/2} \approx 574$.
\end{lemma}

To prove this lemma, we shall prove the following three inequalities:
\begin{equation} \label{claim.1}
  |PUMap\,(\leqslant m,\bN_h)| \leqslant 2^{m} \, (3h)^{-h/2} \, m^{3h/2} \;\; \mbox{ for each } m ;
\end{equation}
\begin{equation} \label{claim.2}
  |UMap(n,\bN_h,0)| \leqslant |PUMap\,(\leqslant 3(n\!+\!h),\bN_h)| \cdot 2^{3(n+h)}\, ;
\end{equation}
and
\begin{equation} \label{claim.3}
  |UMap(n,\bN_h,\ell)| \leqslant |UMap(n\!-\!\ell,\bN_h,0)| \cdot \binom{n\!+\!h}{\ell} \;\; \mbox{ for each } \ell <n.
\end{equation}
Suppose temporarily that we have proved~(\ref{claim.1}), (\ref{claim.2}) and (\ref{claim.3}). 
Then we can use these inequalities in reverse order to complete the proof of the lemma.  For, by~(\ref{claim.3}),
\[ |UMap(n,\bN_h)| = \sum_{\ell} |UMap(n,\bN_h,\ell)| \leqslant 
  \sum_{\ell} |UMap(n-\ell,\bN_h,0)| \cdot \binom{n+h}{\ell}.\]
But, by~(\ref{claim.2}), for each $\ell <n$
\[|UMap(n-\ell,\bN_h,0)| \leqslant |PUMap(\leqslant 3(n+h),\bN_h)| \cdot 2^{3(n+h)}  \]
(where the right hand side does not depend on $\ell$).
Hence 
\begin{eqnarray*}
|UMap(n,\bN_h)|
  & \leqslant &
 \sum_{\ell} |PUMap(\leqslant 3(n+h),\bN_h)|\cdot 2^{3(n+h)} \cdot \binom{n+h}{\ell}\\
& \leqslant & 
2^{4(n+h)} \cdot |PUMap(\leqslant 3(n+h),\bN_h)| \\
& \leqslant &
   2^{4(n+h)} \cdot 2^{3(n+h)} \, (3h)^{-h/2} \,  (3(n+h))^{3h/2} \;\;\;\; \mbox{ by~(\ref{claim.1})}\\
& = &
   2^{7(n+h)}\, 3^{h} \, h^{-h/2} \cdot h^{3h/2} (1+n/h)^{3h/2}\\
 & \leqslant &
 (2^7e^{3/2})^n \, (2^7 3)^h \, h^h,
\end{eqnarray*}
where the last step follows since $1+x \leqslant e^x$ and so $(1+n/h)^{3h/2} \leqslant e^{3n/2}$.
Thus once we have proven~(\ref{claim.1}), (\ref{claim.2}) and~(\ref{claim.3}) we will have proven Lemma~\ref{lem.unimap-nonor}.
\smallskip

\begin{proof}[Proof of inequality~(\ref{claim.1})]
It follows from Euler's formula~(\ref{eqn.Euler}) that each precubic unicellular map in $\bN_h$ has at least $3h-1$ edges,
and (since $h$ is even) each map in $PUMap(m,\bN_h)$ has an odd number of edges, see Lemma 5 of~\cite{nonorientable1}. 
%
Write $h$ as $2j$. By Corollary 8 of~\cite{nonorientable1}, the number of precubic unicellular maps in $\bN_h$ with $m=2k+1$ edges, where $m \geqslant 3h -1$ (or equivalently $k \geqslant 3j-1$), satisfies
\[|PUMap(m,\bN_h)| = c_j \cdot \frac{(2k)!}{6^j \, k! \, (k+1-3j)!}\]
where
\[c_j = 3 \cdot 2^{3j-2} \frac{j!}{(2j)!} \, \sum_{l=0}^{j-1} \binom{2l}{l} 16^{-l}.\]
But
\[ \frac{j!}{(2j)!} = \frac1{(2j)_j} \leqslant j^{-j},\]
and
\[ \sum_{l=0}^{j-1} \binom{2l}{l} 16^{-l} \leqslant \sum_{l\geqslant 0} 2^{2l} 16^{-l} = \sum_{l\geqslant 0} 4^{-l} = \tfrac43 , \]
so
$c_j \leqslant 2^{3j} j^{-j}$. Also
\[ \frac{(2k)!}{k! \, (k+1-3j)!} = \binom{2k}{k} \frac{k!}{(k+1-3j)!} \leqslant 2^{2k} \, k^{3j}.\]
Thus
\begin{eqnarray*}
|PUMap(m,\bN_h)| 
& \leqslant &
2^{3j} j^{-j} \cdot 6^{-j}\,  2^{2k} \, k^{3j}
\;\; \leqslant \;\;
(\tfrac8{6j})^j \, 2^{m-1} \,(\tfrac{m}2)^{3j}\\
& = &
(\tfrac1{6j})^j \, 2^{m-1} m^{3j}
\;\; = \;\;
(3h)^{-h/2}\: 2^{m-1}\, m^{3h/2}.
\end{eqnarray*}
Hence
\begin{eqnarray*}
 |PUMap(\leqslant m,\bN_h)|
 & \leqslant &
 (3h)^{-h/2} \, m^{3h/2} \sum_{m' \leqslant m} 2^{m'-1}\\
 & \leqslant & 
2^{m} (3h)^{-h/2} m^{3h/2},
\end{eqnarray*}
as required.
\end{proof}

\begin{proof}[Proof of inequality~(\ref{claim.2})]
Consider a unicellular $n$-vertex map $M$ in $\bN_h$ (which must have $e(M)=n\!+\!h\!-\!1$ edges) which has no vertices of degree 2.  Given a vertex $v$ of degree at least 4, we may form a new map in the surface by \emph{splitting} $v$ into two vertices, $v$ and $v'$, of degree at least 3, as in Figure~\ref{figure.vertexsplitting}.
\begin{figure}[H]
\centering
\includegraphics[scale=0.4]{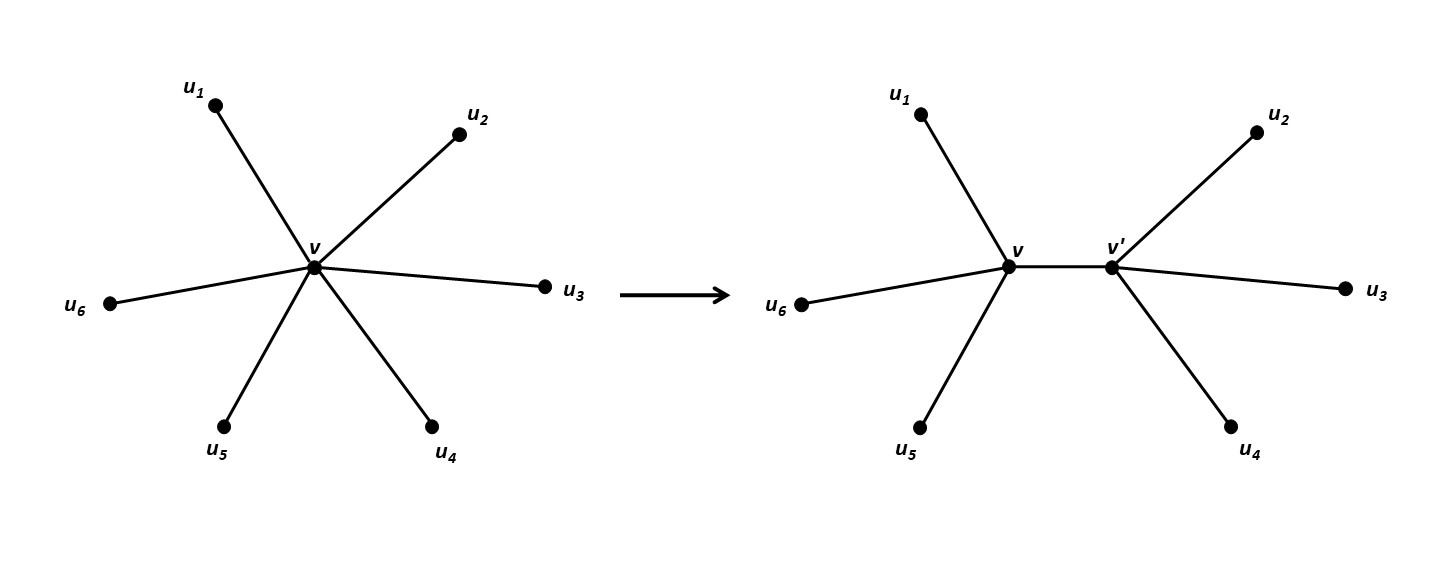}
\caption{Splitting a vertex $v$ of degree greater than three} 
\label{figure.vertexsplitting}
\end{figure}
We can split each vertex of degree greater than three until no such vertices are left.
In every splitting step we add a new vertex and a new edge. To obtain vertices all of degree three from a vertex of degree $d(v)>3$ we need to make exactly $d(v)-3$ vertex splits. In total, summing over all vertices, after making
\[\sum_{\substack{v\in V(G)\\ d(v)>3}} (d(v) -3) \leqslant \sum_{v\in V(G)} d(v) = 2 \, e(M)\]
splits we will have turned $M$ into a unicellular map with each vertex degree 1 or 3, and with at most $3e(M) = 3(n + h-1)$ edges.  Finally, pick an edge, insert a vertex $u$ of degree 2 in this edge, add a leaf vertex adjacent to $u$, and make this vertex the root. This last step adds two edges, so from $M$ we have now constructed a precubic unicellular map $M'$ with less than $3(n+ h)$ edges.  

By deleting the root vertex and suppressing the resulting  vertex of degree 2, and then contracting the new edges in $M'$, we recover the map $M$.  Thus the number of unicellular $n$-vertex maps in $\bN_h$ without vertices of degree 2 is at most $\binom{3(n+h)}{n+h} \leqslant 2^{3(n+h)}$
times the number of unicellular precubic maps in $\bN_h$ with at most $3(n+h)$ edges, as required.
\end{proof}

\begin{proof}[Proof of inequality~(\ref{claim.3})]
Each unicellular n-vertex map in $\bN_h$ with $\ell$ vertices of degree 2 can be obtained from a unicellular map in $\bN_h$ with $n_1=n-\ell$ vertices, and thus with $n_1+h-1$ edges, which has no vertices of degree 2, by inserting $\ell$ vertices of degree 2 into edges.  The number of ways of doing the inserting is at most the number of ways of forming a list of $k=n_1+h-1$ non-negative integers summing to $\ell$, which is
\[ \binom{(k-1)+ \ell}{k -1} =\binom{n\!+\!h\!-\!2}{\ell} \leqslant \binom{n+h}{\ell}. \]
%
Thus the number of unicellular n-vertex maps in $\bN_h$ with $\ell$ vertices of degree 2 is at most $\binom{n+h}{\ell}$ times the number of 
unicellular $n_1$-vertex maps in $\bN_h$ without vertices of degree 2, as required.
\end{proof}

We have now completed the proof of Lemma~\ref{lem.unimap-nonor}.
Next let us handle the case when $h$ is odd, as a corollary of Lemma~\ref{lem.unimap-nonor}.
\begin{lemma} \label{lem.unimap-nonor-both}
There is an $n_0$ such that, for each $n \geqslant n_0$ and $h \geqslant 0$,
\[ |UMap(n,\bN_h)| \leqslant c^{n+h} h^h\]
where 
$c= 2^7 e^{3/2} + 1 \approx 575$.
\end{lemma}

\begin{proof}
By Lemma~\ref{lem.unimap-nonor}, we may assume that $h$ is odd.
Suppose we are given a unicellular map $M$ in $\bN_h$ with $n$ vertices. By picking an edge, inserting a new vertex $u$ to subdivide the edge,
and then attaching to $u$ a loop with signature -1, we may form a unicellular map $M'$ in $\bN_{h+1}$ with $n+1$ vertices. From $M'$ we can recover $M$ if we guess the added vertex $u$. Thus, by Lemma~\ref{lem.unimap-nonor}, letting $c_0$ be the constant there,
\[ |UMap(n,\bN_h)| \leqslant (n+1)\, |UMap(n\!+\!1,\bN_{h+1})| \leqslant (n+1)\, c_0^{n+h+1} (h+1)^{h+1};\]
and the lemma follows since $c>c_0$.
\end{proof}

We may now complete the proof of the non-orientable case of Theorem~\ref{thm.ubmap}, much as in the orientable case.
Let $n_0$ and $c$ be as in Lemma~\ref{lem.unimap-nonor-both}.  Then by Lemmas~\ref{lem.unimap-to-map} and~\ref{lem.unimap-nonor-both}, for all $n \geqslant n_0$ and $h \geqslant 0$,
the number of $n$-vertex simple maps in $\bN_{h}$ is at most $c_0^{n+h} h^h$,
where $c_0 = c \, (2+3\sqrt{2})^2 \approx 2.24 \times 10^5 $.

We have now completed the proofs of both the orientable and the non-orientable cases of Theorem~\ref{thm.ubmap} on maps, which as we saw yields Theorem~\ref{thm.upperbound} on graphs.
\medskip

Now that we have proved both Theorem~\ref{thm.lowerbound} (in Section~\ref{sec.lb}) and Theorem~\ref{thm.upperbound} we can prove Corollary~\ref{corollary_gg}.

\subsection{Proof of Corollary~\ref{corollary_gg}}
\label{proof_corollary_gg}
Recall that $g(n) = n^{1+\eta +o(1)}$ with $g(n) \gg n^{1+\eta}$.
Suppose first that $\eta=0$, so $g(n)=n^{1+o(1)}$ with $g(n) \gg n$.
Then (writing $g$ for $g(n)$ as usual) we have $(n^2/g)^g = g^{(1+o(1))g}$, so by Theorem~\ref{thm.lowerbound}~(b) for some constant $c>0$
\[ |\cA^g_n| \geqslant c^{n+g} (n^2/g)^g\, n! = g^{(1+o(1))g}\,. \]
Also, by Theorem~\ref{thm.upperbound} we have
$\,|\cA^g_n| \leqslant g^{(1+o(1)) g}$.  Thus $|\cA^g_n| = g^{(1+o(1))g}$, as required.

Now suppose that $\eta=\tfrac1{j+1}$ for some $j \in \N$.  Then 
\[ \log (n^2/g) = (1+o(1)) \tfrac{j}{j+1} \log n = (1+o(1)) \tfrac{j}{j+2} \log g\,, \]
so by Theorem~\ref{thm.lowerbound} (c) 
\[ |\cA^g_n| \geqslant (n^2/g)^{(1+o(1))\tfrac{j+2}{j} g} = g^{(1+o(1)) g}\,.\]
Also as before, by Theorem~\ref{thm.upperbound} we have
$|\cA^g_n| \leqslant g^{(1+o(1)) g}$.  Thus again we have $|\cA^g_n| = g^{(1+o(1))g}$, which completes the proof.



\section{Estimating $|\cA^g_n|$, proof of Theorem~\ref{thm.gc-estimate}}
\label{sec.proofthm1}

From the bounds we have already obtained we can very quickly prove part (b) of Theorem~\ref{thm.gc-estimate}.  The great bulk of this section is devoted to proving part (a).

\subsection{Proof of Theorem~\ref{thm.gc-estimate} (b)}

Let 
$g(n) = O(n)$.  By Corollary~\ref{cor.lb} there are constants $c_1>0$ and $n_1$ such that for $n \geqslant n_1$
\[ | \cA^g_n| \geqslant c_1^n g^g n! \; \mbox{ and thus }\; | \tA^g_n| \geqslant c_1^n g^g. \]
By Theorem~\ref{thm.upperbound} there is a constant $c_2$ such that for all $n \geqslant 1$
\[|\tA^g_n| \leqslant c_2^n g^g \; \mbox{ and thus }\;  | \cA^g_n| \leqslant c_2^n g^g n!.\]
It follows that $\,|\cA^g_n|= 2^{\Theta(n)} g^g n!\,$ and $\,|\tA^g_n|= 2^{\Theta(n)} g^g$, as required.


\subsection{Proof of Theorem~\ref{thm.gc-estimate} (a)\, (on growth constant $\gamma_{\cP}$)}
\label{sec.gc}

In this subsection we will prove Theorem~\ref{thm.gc-estimate} (a), which says essentially that when $g(n)=o\left(n/\log^3n\right)$ the class $\cE^g$ is not too much larger than $\cP$.
We use the notation $R_n \inu \cA^g$ to mean that the random graph $R_n$ is sampled uniformly from the graphs in $\cA^g_n$. For most of the proof we assume that $g$ is non-decreasing.
We first show that for `most' integers $n$, the random graph $R_n \inu \cA_n^g$ whp has
linearly many leaves, and deduce that for these integers $n$ whp $R_n$ has small maximum degree.  Then we can use the following `planarising' result \cite[Theorem 4]{Planarization}.
Given a graph $G$, a \emph{planarising edge-set} is a set of edges such that deleting these edges from $G$ leaves a planar graph.

\begin{lemma} \cite{Planarization}
\label{lem.edgeplane}
For all $n \geqslant 2$ and $h \geqslant 0$, every connected graph in $\cE_n^h$ with maximum degree at most $\Delta$ has a planarising edge-set of size at most $4 \sqrt{h(n+h-2) \Delta}$. 
\end{lemma}

We next give a sequence of five 
lemmas which yield a bound on maximum degree, and allow us to use Lemma~\ref{lem.edgeplane} to prepare for the final steps in the proof of Theorem~\ref{thm.gc-estimate}(a).
In these lemmas we assume that we are given a non-decreasing genus function $g$
satisfying $g(n) = O(n/\log n)$, and we are given a constant $0< \eps <1$. 
We start by showing that for `most' positive integers $n$, the set $\cA_{n+1}^{g}$ is not much bigger than $\cA_{n}^{g}$.
Given $0<\delta <1$ we say that a set $I \subseteq {\mathbb N}$ has \emph{lower (asymptotic) density at least} $\delta$ if for all sufficiently large $n \in {\mathbb N}$ we have $|I \cap [n]| \geqslant \delta n$.
\begin{lemma}\label{lemma:setS}
Let $g$ 
be non-decreasing and satisfy $g(n) = O(n/\log n)$; and let $0< \eps <1$.  Then there exists a constant $c_1=c_1(g,\eps)$ such that the set $I^*(g,\eps)$
of integers $n \geqslant 1$ for which
\begin{equation} \label{eqn.slowgrowth} \notag
\left|\cA_{n+1}^{g}\right| \leqslant c_1 \, (n+1)  \left|\cA_{n}^{g}\right|
\end{equation}
has lower density at least $1-\eps$.
\end{lemma}
\begin{proof}
By Theorem~\ref{thm.upperbound} (and the comment following it), there is a constant $c_0 >1$ such that
\begin{equation} \label{eqn.start}
|\cA_n^g| \leqslant c_0^n \, n! \;\; \mbox{ for all } n \geqslant 1.
\end{equation}
We shall see that we may take $c_1= c_0^{1/\eps}$.
Let $n \in {\mathbb N}$, and suppose for a contradiction that there are more than $\eps n$ integers $m\in [n]$ such that
\[ \left|\cA_{m+1}^{g}\right|\geqslant c_1 \, (m+1) \left|\cA_m^{g}\right| \, . \]
By inequality~(\ref{eqn.growth}), for all $m\in \mathbb{N}$ we have (since $g(m+1) \geqslant g(m)$)
\[ \left|\cA_{m+1}^{g}\right|\geqslant \left|\cA_{m+1}^{g(m)}\right| \geqslant 2m\, \left|\cA_m^{g}\right| \geqslant (m+1) \left|\cA_m^{g}\right|. \]
Hence
\[ \left| \cA_{n}^{g} \right| \; > \; c_1^{\eps n} \, n! \; = \; c_0^n \, n!\]
contradicting~(\ref{eqn.start}).
\end{proof}
From now on we shall let $I^*=I^*(g,\eps)$ be as in the last lemma.
\begin{lemma}\label{lemma:alphanbound}
Let $g$ 
be non-decreasing and satisfy $g(n) = O(n/\log n)$; and let $0< \eps, \, p <1$.
Let $I^*=I^*(g,\eps)$ be as in Lemma~\ref{lemma:setS}.
Let $R_n \inu \cA^{g}$.
Then there exist $\alpha>0$ and $n_0\in \mathbb{N}$ such that for all $n\geqslant n_0$ with 
$n \in I^*$ 
\begin{equation} \notag
  \tilde{p}(n): =  \; \mathbb{P}(R_n \text{ has at least } \alpha n \text{ leaves}) \geqslant p \text{ .}
\end{equation}
\end{lemma}
\begin{proof}
Let $\alpha = \frac{(1-p)}{2c_1}$, where $c_1$ is as in Lemma~\ref{lemma:setS}.
To prove the lemma we will show that $\tilde{p}(n) \geqslant p$ for sufficiently large $n \in I^*$.
We do this by constructing from each graph $G \in \cA_n^g$ with few leaves many graphs $G' \in \cA^g_{n+1}$, with little double counting.

Let $n \in I^*$ and let $G\in \cA_n^{g}$ have less than $\alpha n$ leaves. There are exactly $(1-\tilde{p}(n))\, \left|\cA_n^{g}\right|$ such graphs. To construct a graph $G'\in\cA_{n+1}^{g}$ from $G$, we first pick one of the vertices in $[n+1]$, $v$ say. 
There are $n+1$ choices for this. We now put a copy $\widehat{G}$ of $G$ on the vertex set $[n+1]\setminus \{v\}$ in such a way that the order-preserving bijection from $[n]$ to $[n+1]\setminus \{v\}$ is an isomorphism from $G$ to $\widehat{G}$. We form $G'$ by adding the vertex $v$ to $\widehat{G}$ as a leaf incident to some vertex $y\in [n+1]\setminus\{v\}$. Since there are $n$ choices for $y$, in total we make
$(1-\tilde{p}(n))\, \left|\cA_n^{g}\right|(n+1)n$ constructions of graphs $G' \in \cA_{n+1}^g$.

How often is each graph $G'\in \cA_{n+1}^{g}$ constructed? To get back to $G$ from $G'$, we just need to find the vertex $x$ (which is a leaf in $G'$),
delete it, and then move the vertex set from $[n+1]\setminus\{x\}$ to $[n]$ using the order-preserving bijection. How many choices for $x$ are there? There are at most $\lceil\alpha n\rceil$ leaves in $G'$, so each graph $G'$ is constructed at most $\lceil\alpha n\rceil$ times. We thus have
\begin{equation}\label{equation:Galpha} \notag
\left|\cA_{n+1}^{g}\right|\geqslant \left|\cA_n^{g}\right|\, (1-\tilde{p}(n))\, \frac{n^2}{\lceil \alpha n\rceil}\text{ .}
\end{equation}
But $\left|\cA_{n+1}^{g}\right|\leqslant c_1  (n+1) \left|\cA_n^{g}\right| $ since $n \in I^*$, so we obtain
\begin{equation} \notag
\left|\cA_n^{g}\right|\, (1-\tilde{p}(n)) \, \frac{n^2}{\lceil \alpha n\rceil} \leqslant \left|\cA_{n+1}^{g}\right| \leqslant c_1  (n+1) \left|\cA_n^{g}\right| \, .
\end{equation}
Hence
\[ 1-\tilde{p}(n) \leqslant \tfrac{\lceil \alpha n\rceil}{n} \tfrac{n+1}{n} \, c_1  = \alpha  c_1 + O(\tfrac1{n}) = \tfrac{1-p}2 + O(\tfrac1{n}),\]
and so
\[ \tilde{p}(n)\geqslant \tfrac{1+p}2 + O(\tfrac1{n}) \geqslant p\]
for $n$ sufficiently large, as required.
\end{proof}

We have now seen that, as long as $g(n)=O(n / \log n)$ and $g$ is non-decreasing, for $n \in I^*$ the random graph $R_n \inu \cA^g$ `often' has linearly many leaves.
We now use this result to show that `often' the maximum degree $\Delta(R_n)$ is small. 
%
In order to be able to control 
the maximum degree $\Delta(R_n)$ when $n \in I^*$ we shall use two further preliminary lemmas, Lemmas~\ref{lem.fewpend} and~\ref{lem.Delta}.
Both the lemmas are generalisations of results in~\cite{MaxDegree}.
Lemma~\ref{lem.fewpend} concerns the maximum number of leaves adjacent to any vertex.  We spell out a proof here for completeness, though the proof closely follows the proof of Lemma 2.2 in~\cite{MaxDegree}.  See Theorem~4.1 in~\cite{MaxDegree} for a related sharper and more general result.
\begin{lemma} \label{lem.fewpend}
Let $\cG$ be a class of graphs which is closed under detaching and re-attaching any leaf, and let $R_n \inu \cG$. Then whp each vertex in $R_n$ is adjacent to at most $2 \log n /\log\log n$ leaves.
\end{lemma}

\begin{proof}
Let $k$ be a positive integer, and for each $n \in \N$ let $\cB_n$ be the set of graphs $G \in \cG_n$ such that vertex 1 is adjacent to at least $k$ leaves. 
We claim that
\begin{equation} \label{claim.leaves}
\pr( R_n \in {\cal B}_n) \leqslant 1/k!
\end{equation} 
which will yield the lemma, since it shows that the probability that $R_n$ has some vertex adjacent to at least $k$ leaves 
is at most $n/k!$.

Let us prove the claim~(\ref{claim.leaves}).
For each graph $G \in \cB_n$, consider the $k$ least pendant vertices $u_1,\ldots,u_k$ adjacent to vertex $1$, remove the edges incident with these vertices $u_i$, and arbitrarily re-attach each vertex $u_i$ to one vertex of $G$ other than $u_{i+1},\ldots,u_k$.  Then each graph $G'$ constructed is in $\cG_n$, and the number of constructions is at least $|\cB_n|\, (n\!-\!1)_k$.  (Recall that $(x)_k$ denotes the `falling factorial'
$x(x-1) \cdots (x-k+1)$.)

How often can each graph $G' \in \cG_n$ be constructed? We may guess the set of $k$ vertices $u_i$ and then we know the original graph $G$. Thus each graph $G'$ is constructed at most $\binom{n-1}{k}$ times.  Hence
\[|\cG_n| \geqslant |\cB_n| \ (n\!-\!1)_k/\tbinom{n-1}{k}= |\cB_n| \ k!\]
and so
\[ \pr ( R_n \in \cB_n ) = |\cB_n|/|\cG_n| \leqslant 1/k!\]
as required for~(\ref{claim.leaves}).
\end{proof}


Let $\cS$ 
be the set of graphs $G$ such that if $G$ has $n$ vertices then each vertex is adjacent to at most $2\log n/\log\log n$ leaves (where $\cS$ is for {\bf s}mall number of leaves).  Since $\cA_n^g$ is closed under detaching a leaf and re-attaching it, by Lemma~\ref{lem.fewpend}
we have $R_n \in \cS$ whp. Now, given $0<\alpha<1$, let $\cL^\alpha$ be the set of graphs $G$ which have at least $\alpha\, v(G)$ leaves.
The next lemma concerns both $\cS$ and $\cL^\alpha$.
\begin{lemma} \label{lem.Delta}
Let $0<\alpha<1$, let $b=b(n) = \lceil (8/\alpha)\, \log n \rceil$, and let
\[\cB = \{ G \in \cL^\alpha  \cap \cS : \Delta(G) \geqslant b(n) \mbox{ where } n=v(G) \}. \]
There is a function $\eta(n)=o(1)$ as $n \to \infty$ such that the following holds: for all $n \in \N$ and all surfaces $S$, the random graph $R^S_n \inu \cE^S$ 
satisfies $\pr(R^S_n \in \cB) \leqslant \eta(n)$.
\end{lemma}
\noindent
(Observe that $\eta(n)$ does not depend on the surface $S$.)
The following proof is adapted from the proof of Theorem~1.2 in~\cite{MaxDegree}.
\begin{proof}
For each surface $S$ let $\cB^S = \cB \cap \cE^S$.  The idea of the proof is similar to some earlier proofs:
from each graph in $\cB^S_n$ we can build many graphs in $\cE^S_n$ with little double counting, so we cannot start with many graphs in $\cB^S_n$.
Let $a=a(n) = \lfloor 2 \log n \rfloor$. Let $\eta(n)= n/2^{a-1}$, so $\eta(n) = o(1)$.  Let $n_0$ be sufficiently large that for each $n \geqslant n_0$ we have $a \geqslant 3$ and $\alpha n - 2 \log n / \log\log n - a \geqslant \tfrac12 \alpha n $.  Assume that $ n \geqslant n_0$, and let $S$ be any surface.

Here is the construction. 
Let $G \in \cB^S_n$, and fix an embedding of $G$ in $S$.  Let $v$ be a vertex with degree at least $b$. The embedding gives a clockwise order on the neighbours of $v$: list them in this order as $v_1,v_2,\ldots,v_{d}$ where $d \geqslant b$ is the degree of $v$ and where $v_d$ is the largest of the numbers $v_1,\ldots,v_d$.
Choose an arbitrary ordered list of $a$ distinct pendant vertices with none adjacent to $v$, say $u_1,\ldots,u_a$. Finally choose an arbitrary subset of $a$ of the $d \geqslant b >a$ vertices $v_i$, which we may write as $v_{i_1},\ldots,v_{i_a}$ where $i_1<i_2< \cdots < i_a$.

Now for the graph part. Delete each edge incident to $v$, and each edge incident to one of the chosen pendant vertices $u_i$.  For each $i=1,\ldots,a$, join $v$ to $u_i$ and join $u_i$ to $u_{i+1}$ (where $u_{a+1}$ means $u_1$).  Thus we have formed a wheel around $v$. For each $j=1,\ldots,a$, join $u_j$ to each of $v_{i_j},v_{i_j +1},\ldots,v_{i_{j+1}-1}$ (where $i_{a+1}$ means $i_1$).
This completes the construction.  It is easy to see that each graph $G'$ constructed is in $\cE^S_n$.

For each $G \in {\cal B}_n$, we make at least $(\alpha n - 2 \log n/ \log\log n)_{a} \geqslant (\tfrac12 \alpha n)^{a}$
choices for the list of pendant vertices $u_1,\ldots,u_a$, and at least $\binom{b}{a} \geqslant \left(\frac{b}{a}\right)^{a}$ choices for the subset of the neighbours of $v$. Thus the total number of constructions is at least
\[|{\cal B}_n| \ \left( \frac{\alpha n}{2} \cdot \frac{b}{a}\right)^{a} \geqslant |\cB_n| \, (2n)^{a}.\]

Now consider the double counting.  How many times can a given graph $G' \in \cE^S_n$ be constructed?  Guess the vertex $v$.  Find the largest `second neighbour' of $v$: this is $v_d$. This determines $u_a$ (the unique neighbour of $v$ adjacent to $v_d$).  Now guess which of the two common neighbours of $v$ and $u_a$ is $u_1$ (the other is $u_{a-1}$).  Now we know each of $u_1, u_2,\ldots,u_a$.  Next guess the original neighbours of these vertices.  This determines the original graph $G$ completely.  So the embedding is determined, and in particular the order $v_1,\ldots,v_d$ of the neighbours of $v$.  But for each $j=1,\ldots,a-1$ the vertex $v_{i_j}$ is the earliest vertex in this list adjacent in $G'$ to $u_j$, and $v_{i_a}$ is the earliest vertex in this list which is adjacent in $G'$ to $u_a$ and is also after $v_{i_{a-1}}$ in the cyclic order. Hence we know $v_{i_1}, v_{i_2},\ldots, v_{i_a}$, and all choices have been determined.  Thus $G'$ is constructed at most  $n \cdot 2 \cdot n^a = 2 n^{a+1}$ times.
%
Hence
\[ |\cA_n| \geqslant |\cB_n| (2n)^{a} / (2n^{a+1})\]
and so
\[ \pr[R_n \in \cB_n] = |\cB_n| / |\cA_n| \leqslant  n/2^{a-1} = \eta(n)\,,\]
which completes the proof of the lemma. 
\end{proof}

We can now obtain the desired bound on the maximum degree.
\begin{lemma} \label{lem.delta}
Let $g$ 
be non-decreasing and satisfy $g(n) = O(n/\log n)$; and let $0< \eps <1$.
Let $I^*=I^*(g,\eps)$ be as in Lemma~\ref{lemma:setS}.
Let $R_n \inu \cA^{g}$.
Then there exists $0<\alpha<1$ such that, setting $b=b(n) = \lceil (8/\alpha)\, \log n \rceil$ as in Lemma~\ref{lem.Delta}, for all sufficiently large $n$ in $I^*$ we have
\begin{equation} \label{eqn.Delta2}
\pr(\Delta(R_n) < b) \geqslant \tfrac{1}{2}\, .
\end{equation}
\end{lemma}
\begin{proof}
By Lemma \ref{lemma:alphanbound} with $p=\tfrac23$ there exists a constant $\alpha >0$ such that for all sufficiently large $n \in I^*$ we have $\mathbb{P}(R_n \in \cL^\alpha) \geqslant \tfrac23$ (where $\cL^{\alpha}_n$ is the set of graphs $G$ on $[n]$ with at least $\alpha n$ leaves).
Thus by Lemma~\ref{lem.fewpend}
\[ \pr(R_n \not \in \cL^\alpha \cap \cS) \leqslant \tfrac13 + o(1).\]
Hence by Lemma~\ref{lem.Delta},
for $n\in I^*$
\[ \pr(\Delta(R_n) \geqslant b) \leqslant
\pr\left( (R_n \in \cL^\alpha \cap \cS)) \land (\Delta(R_n) \geqslant b)\right) + \pr(R_n \not\in \cL^\alpha \cap \cS)) \leqslant \tfrac13 +o(1)\,,
\]
which gives~(\ref{eqn.Delta2}).
\end{proof}

Lemma~\ref{lem.delta} allows us to use the planarising result  Lemma~\ref{lem.edgeplane} to upper bound the sizes of the sets $\cA_n^g$, first for $n \in I^*$ in Lemma~\ref{lemma_sepsilongc} and then for all $n$ in Lemma~\ref{lemma:gcgammalepsilon} (still assuming that $g$ is non-decreasing).
\begin{lemma}\label{lemma_sepsilongc}
Let $g$ 
be non-decreasing and satisfy $g(n)=o\left(n / \log^3 n \right)$;
and let $0< \eps <1$. 
Let $I^*=I^*(g,\eps)$ be as in Lemma~\ref{lemma:setS}.
Then as $n\rightarrow \infty$ with $n$ in $I^*$
\begin{equation} \notag
\left|\cA_n^{g}\right| \leqslant (1+o(1))^n \, \gamma_{\mathcal{P}}^n \, n!  \, .
\end{equation}
\end{lemma}
\begin{proof}
Assume that $n \in I^*$ and that $n$ is sufficiently large that~(\ref{eqn.Delta2}) holds, so at least $\tfrac12$ of all graphs in $\cA_n^{g}$ have maximum degree at most $c_2 \log n$. 
Define $c_3 =5 \, \sqrt{c_2}$.
Let $G\in \cA_n^{g}$ have $\Delta(G) \leqslant c_2 \log n$. Then by Lemma~\ref{lem.edgeplane} there exists a set of at most $t:= c_3 \sqrt{ng\log n}$ edges such that deleting these edges leaves a planar graph $G'$. How often is each planar graph $G'$ constructed?  Note the crude bound that for all integers $2 \leqslant j \leqslant k$
\[ \sum_{i=0}^j \binom{k}{i} \leqslant k^j .\]
Thus there are at most 
\[\sum_{i=0}^t  \binom{\binom{n}{2}}{i} \leqslant n^{2t} \]
choices for which set of at most $t$ edges to add to $G'$ to obtain $G$. Hence each graph $G'$ is constructed at most $n^{2t}$
times. Since at least half of all graphs in $\cA_n^{g}$ have maximum degree at most $c_2 \log n$ we have
\begin{equation} \notag
\left|\cA_n^{g}\right| \leqslant 2 \, n^{2t} \left|\cP_n\right| = (1+o(1))^n \left|\cP_n\right| = (1+o(1))^n \gamma_{\mathcal{P}}^n \cdot n!
\end{equation}
as required.
\end{proof}

We have now found a bound on the size of $|\cA_n^{g}|$ for all $n$ in the set $I^*$; and using this, we next prove an upper bound on $|\cA_n^{g}|$ for \emph{all} $n\in \mathbb{N}$.
\begin{lemma}\label{lemma:gcgammalepsilon}
Let $g$ 
be non-decreasing and satisfy $g(n) = o(n/\log^3 n)$; and let $0< \eps <1$.
Then as $n\rightarrow \infty$ (without any restriction)
\begin{equation} \notag
|\cA_n^{g}|\leqslant (1+o(1))^n \, \gamma_{\mathcal{P}}^{(1+\eps)n} \, n! \text{ .}
\end{equation}
\end{lemma}
\begin{proof}
Now we let $I^*=I^*(g,\tfrac12 \eps)$, as in Lemma~\ref{lemma:setS}.
By Lemma \ref{lemma_sepsilongc}, as $n\rightarrow \infty$ with $n\in I^*$ 
\begin{equation} \notag
|\cA_n^{g}| \leqslant (1+o(1))^n \, \gamma_{\cP}^n \, n! \, .
\end{equation}
All that is left to show is that this is also satisfied for all $n\not\in I^*$. To do so, suppose that $n \not\in I^*$.  Since $(\tfrac12 \eps)(1+\eps)n< \eps n$ and the interval $[n, (1+\eps)n]$ contains at least $\eps n$ integers, there exists an $m\in I^*$ such that
$n<m\leqslant (1+\eps)n$. Furthermore, recall that by inequalities~(\ref{eqn.grs}) for all $n \geqslant 1$
\begin{equation} \notag
\left|\cA_{n+1}^{g}\right| \geqslant 2n\, \left|\cA_n^{g}\right| \geqslant (n+1) \left|\cA_n^{g}\right|\text{ .}
\end{equation}
From this, it follows that
\begin{equation}
\begin{split}
\left|\cA_n^{g}\right| &\leqslant \frac{1}{m (m-1) \cdots (n+1)} \left|\cA_m^{g}\right|\notag\\
&\leqslant \frac{1}{ m (m-1) \cdots (n+1)} \cdot (1+o(1))^m \, \gamma_{\cP}^m  \, m!\notag\\
& = (1+o(1))^n \, \gamma_{\mathcal{P}}^{(1+\eps)n} \, n!
\end{split}
\end{equation}
and this completes the proof.
\end{proof}

We are at last in a position to complete the proof of Theorem~\ref{thm.gc-estimate} (a).  Note that we do not assume that $g$ is non-decreasing.

\begin{proof}[Proof of Theorem~\ref{thm.gc-estimate} (a)]
Define the function $g^+=g^+(n)$ by setting $g^+(n)= \max\{g(1),\ldots,g(n)\}$. Then $g^+(n) = o(n/\log^3n)$ and $g^+$ is non-decreasing.
Since $|\cA_n^g| \leqslant  |\cA_n^{g^+}|$ for each $n$, by Lemma \ref{lemma:gcgammalepsilon} applied to $g^+$
\begin{equation}\label{equation:limsup} \notag
\limsup_{n\rightarrow \infty} \left(|\cA_n^{g}| / n! \right)^{1/n} \leqslant \gamma_{\mathcal{P}}\, .
\end{equation}
But since also $\cP_n\subseteq \cA_n^{g}$ we have 
\begin{equation}\label{equation:liminf} \notag
\liminf_{n\rightarrow \infty} \left(|\cA_n^{g}| / n! \right)^{1/n} \geqslant \lim_{n\rightarrow \infty} \left(|\cP_n| / n! \right)^{1/n} = \gamma_{\cP} \, ,
\end{equation}
which completes the proof.
\end{proof}



\section{Estimating $|\cF^h_n|$}
\label{subsec.estFhn}

Recall that $\cF^h_n$ is the set of graphs on $[n]$ such that \emph{every} cellular 
embedding is in a surface with Euler genus at most $h$. 
When considering large values of $h$ the separate factor $n!$ in the bound 
$\left|\cF_n^{h}\right| \geqslant c^{n+h} \,(n^2/h)^h \: n!$ given in Theorem~\ref{thm.lowerboundab} (b) is not helpful.  In this short section we give an estimate of $|\cF^h_n|$ valid for all $n$ and all relevant values of $h$, which immediately yields the estimate~(\ref{eqn.largegF2}).
%
\begin{proposition}
 \label{thm.estF}
There are constants $0<c_1<c_2$ such that, for all $n \geqslant 1$ and $0 \leqslant h \leqslant n^2$
\[  \left( \frac{c_1 n^2}{n\!+\!h} \right)^{n+h} \leqslant  \left|\cF_n^{h}\right| \leqslant \left( \frac{c_2 n^2}{n\!+\!h} \right)^{n+h}\,. \]
\end{proposition}
\begin{proof}
\emph{Upper bound.} 
Suppose first that $n+h \leqslant \frac12 \binom{n}{2}$.  For each graph $G$ in $\cF^h_n$ we have $e(G) \leqslant h+n-1$, so
\begin{eqnarray*}
|\cF^h_n| & \leqslant &
\sum_{i \leqslant n+h-1}  \binom{\binom{n}2}{i} \;\;
 \leqslant  \;\; (n+h) \, \binom{\binom{n}2}{n+h-1} \\
& \leqslant & (n+h) \, \left(\frac{e \, n^2 }{2(n\!+\!h)}\right)^{n+h-1} \;\;\; \mbox{ using }\; \tfrac{n-1}{n+h-1} \leqslant \tfrac{n}{n+h}\\
& = & \frac{1}{e}\, \frac{(n+h)^2}{n^2 \,2^{n+h-1}} \left(\frac{e n^2}{n\!+\!h}\right)^{n+h}\;
\leqslant \;\;
\left(\frac{e n^2}{n\!+\!h}\right)^{n+h}\,.
\end{eqnarray*}
Suppose now that $n+h > \frac12 \binom{n}{2}$ and $h \leqslant n^2$.   Then
\[  |\cF^h_n|  \leqslant 2^{\binom{n}{2}} \leqslant  \left(\frac{8\, n^2}{n+h}\right)^{n+h}  \]
since
\[ \left(\frac{8\, n^2}{n+h}\right)^{n+h} \geqslant \left(\frac{8\, n^2}{n+n^2}\right)^{n+h} \geqslant 4^{n+h} \geqslant 2^{\binom{n}{2}}\,.\]
Taking $c_2$ as 8 
completes the proof of the upper bound.
\medskip

\noindent
\emph{Lower bound.}
Recall that $\cC^h$ is the class of connected graphs in $\cF^h$. 
As in the proof of Theorem 2 (b), for $n \geqslant 15$ and $0 \leqslant h \leqslant \frac13 n^2$ we have
\[ \left|\cC_n^{h}\right| \geqslant \; n^{n-2} \, \left(\frac{n^2 -3n}{2\,(n+h)}\right)^{h} \geqslant \;  n^{-2} 2^n \left(\frac{n^2 -3n}{2\,(n+h)}\right)^{n+h}  \geqslant \; \left(\frac{n^2}{3\,(n+h)}\right)^{n+h}. \]
But the final bound here is less than 1 if $h > \frac13 n^2$, so
\[ \left|\cF_n^{h}\right| \geqslant \; \left|\cC_n^{h}\right| \geqslant \; \left(\frac{n^2}{3\,(n+h)}\right)^{n+h} \]
for $n \geqslant 15$ and all $h \geqslant 0$.
The lower bound now follows easily: we may set $c_1=\frac1{14}$, since then
$\frac{c_1 n^2}{n\!+\!h} \leqslant c_1 n \leqslant 1$ for all $1 \leqslant n \leqslant 14$ and $h \geqslant 0$.

%
\end{proof}

\section{The hereditary graph classes $\hered(\cA^g)$ and $\hered(\cF^g)$}\label{sec.hered}
In this section we prove Theorem~\ref{cor.excess}, which shows that the radius of convergence of $\rho(\hered(\cF^g))$ drops to 0 when $g(n) \gg n/\log n$; and this also holds for $\rho(\hered(\cA^g))$ since $\hered(\cF^g)\subseteq \hered(\cA^g)$.
Recall that by~(\ref{eqn.cr}), $\hered(\cF^g)$ is the class of graphs $G$ such that for each subset $W$ of vertices we have $\crank(G[W]) \leqslant g(|W|)$.
We shall deduce Theorem~\ref{cor.excess} from
Lemma~\ref{lem.excessnew} below, which gives an explicit lower bound on $|\hered(\cF^g)_n|$ for a suitable genus function $g$. 
We also give a corresponding larger explicit lower bound on $|\hered(\cA^g)_n|$ in Lemma~\ref{thm.excessA}, though that cannot tell us more about the radius of convergence.
Finally we prove Proposition~\ref{prop.heredsmall}, which shows that in some interesting cases $\hered(\cA^g)$  is much smaller than $\cA^g$. In Section~\ref{subsec.stronghered}, we consider `certifiably hereditarily embeddable' graphs.

Recall that, given a class $\cB$ of graphs, we say that a graph $G$ is \emph{hereditarily} in $\cB$
if each induced subgraph of $G$ is in $\cB$; and we call the class of graphs which are hereditarily in $\cB$ the \emph{hereditary part} of $\cB$, denoted by $\hered(\cB)$.  Clearly $\hered(\cB) \subseteq \cB$.
If for example the genus function $g$ 
satisfies $g(n)=0$ for $n \leqslant 5$ and $g(6)=2$, and $G$ is the complete graph $K_5$ plus a leaf, then $G \in \cE^g$ but $G \not\in \hered(\cE^g)$. Of course $\cF^g \subseteq \cA^g$, and thus $\hered(\cF^g) \subseteq \hered(\cA^g)$ (as we noted above).  The containment can be strict.
We saw earlier that if $g$ is identically 0 then $\cA^g = \cP$ and $\cF^g$ is the 
class of forests.  It follows that, if $g$ is identically 0,
then $\hered(\cA^g) = \cP$
and $\hered(\cF^g)$ is the class of forests.

Theorem~\ref{cor.excess} will follow quickly from the next lemma, which gives an explicit lower bound on $|\hered(\cF^g)_n|$ for a suitable genus function $g$.
\begin{lemma}
\label{lem.excessnew}
Let the genus function $g$
satisfy $g(n) \to \infty$ and $g(n)/n \to 0$ as $n \to \infty$; and suppose that there is an $n_0$ 
such that for $n \geqslant n_0$, $g(n)$ is non-decreasing and $g(n)/n$ is non-increasing.  Then
\begin{equation} 
 |\hered\,(\cF^g)_n|\geqslant 
n!\: g^{(1+o(1))\,g/2}\,.
\end{equation}
\end{lemma}

We shall 
prove Lemma~\ref{lem.excessnew} below, but first let us use it to deduce Theorem~\ref{cor.excess}, and then deduce the results~(\ref{eqn.rhopos-hered}) and~(\ref{eqn.rhotpos-hered}).

\begin{proof}[Proof of Theorem~\ref{cor.excess} using Lemma~\ref{lem.excessnew}]
Let the function $f(n) = \max \{1, \log\log n\}$ for $n \in \N$.
Further, let $g_1(n) = \min\{g(n), n/ f(n)\}\,$; and note that $\,g_1(n) \leqslant g(n)$, $\,g_1(n) \gg n/\log n$ and $\,g_1(n)=o(n)$. 
Let $g_2(n) = \min\{g_1(k) : k \geqslant n\}\,$; and note that $\,g_2(n) \leqslant g_1(n)$, $\,g_2(n) \gg n/\log n$ and $\,g_2(n)$ is non-decreasing.
Let $n_0 \in \N$ be such that $g(n) \geqslant 1$ for all $n \geqslant n_0$.  Let $g_3(n) = g_2(n)$ for $n < n_0$, and for $n \geqslant n_0$ let $g_3(n) = n\, \min\{g_2(k)/k : n_0 \leqslant  k \leqslant n\}\,$.  Note that
$\,g_3(n) \leqslant g_2(n)$, $\,g_3(n) \gg n/\log n$ and $\,g_3(n)/n$ is non-increasing for $n \geqslant n_0$. Also $g_3(n)$ is non-decreasing for $n \geqslant n_0$, since $g_2(n+1) \geqslant g_2(n) \geqslant g_3(n)$ and so
\[ g_3(n\!+\!1) = \min \{ \tfrac{n+1}{n} g_3(n),\, g_2(n\!+\!1)\}  \geqslant g_3(n) \,. \]
It follows that $\, g_3(n) \leqslant g(n)$, $\, g_3(n) \gg n/\log n$ and $g_3$ satisfies the conditions in Lemma~\ref{lem.excessnew}.
Hence, by Lemma~\ref{lem.excessnew} applied to $g_3$,
\[ |\hered\,(\cF^g)_n|\geqslant  |\hered\,(\cF^{g_3})_n|\geqslant 
n!\: g_3^{(1+o(1))\,g_3/2}\,; \]
and so
\[ \left(|\hered(\cF^{g})_n|/n! \right)^{1/n} \to \infty \;\; \mbox{ as } n \to \infty \,, \]
as required.
\end{proof}

Let us spell out the proofs of the results~(\ref{eqn.rhopos-hered}) and~(\ref{eqn.rhotpos-hered}) which are presented immediately after Theorem~\ref{cor.excess}.  By Theorem~\ref{cor.excess}
and the result that $\rho(\cA^g)>0$ if $g(n)=O(n/\log n)$ (which is part of~(\ref{thm.gc-estimate})) we immediately obtain~(\ref{eqn.rhopos-hered}). In the unlabelled case, the first part of~(\ref{eqn.rhotpos-hered}) follows directly from the first part of~(\ref{eqn.rhotpos}), and the second part from the second part of~(\ref{eqn.rhopos-hered}).
\bigskip

Given $k=k(n)$, let $\cZ^k$ be the class of graphs 
$G$ such that if $v(G)=n$ then $G$ is a subdivision of a cubic graph $H$ with $k$ vertices, 
such that in $G$ each of the $\frac32 k$ edges of $H$ is subdivided at least $s$ times, where $s= s(n) = \lfloor \frac{2(n-k)}{3k} \rfloor$.  (We could consider cubic pseudographs $H$ weighted by their compensation factor, but this added complication would not yield a significant improvement.)  To prove Lemma~\ref{lem.excessnew}, we will use two further lemmas, namely Lemma~\ref{lem.XsubsetheredF}, in which we show that for a suitable choice of $k=k(n)$ we have $\cZ^k \subseteq \hered(\cF^g)$; and Lemma~\ref{thm.excessagain}, in which we show that $\cZ^k_n$ is large.

\begin{lemma}\label{lem.XsubsetheredF}
Let the genus function $g$ satisfy $g(n) \to \infty$ and $g(n)/n \to 0$ as $n \to \infty$; and suppose that there is an $n_0$ 
such that for $n \geqslant n_0$, $g(n)$ is non-decreasing and $g(n)/n$ is non-increasing.  Let $k = k(n) = 2 \lfloor n g / (2n+3g) \rfloor$. Then
\begin{equation} 
\cZ^k\subseteq \hered\,(\cF^g) \;\; \mbox{ for $n$ sufficiently large}\,.
\end{equation}
\end{lemma}

Before proving Lemma~\ref{lem.XsubsetheredF} let us make some observations about the cycle rank $\crank(G)$
for a pseudograph $G$.  (Recall from~(\ref{eqn.cr}) that $\egmax(G) = \crank(G)$.)
A key observation is that if $G'$ is obtained from $G$ by adding a leaf or subdividing an edge then $\crank(G')= crank(G)$.  If $C$ is a cycle then $\crank(C)=1$,
and so if $G$ has exactly one cycle then $\crank(G)=1$. If $G$ has components $G_1,\ldots,G_\kappa$ then $\crank(G) = \sum_{i=1}^{\kappa} \crank(G_i)$. 
%
  
Let $H$ be the \emph{core} of $G$, obtained by repeatedly deleting any leaves.  Note that we do not 
restrict attention to the complex part of $G$ (consisting of the components with more than one cycle), so the core may contain components which are cycles.
The \emph{kernel} $K$ of $G$ is the pseudograph obtained from the core $H$ by suppressing all vertices of degree 2, except that a component of $H$ which is a cycle yields a component of $K$ which is a single vertex with a loop. Thus any component of $G$ with exactly one cycle becomes a vertex with a loop in $K$. Then $\crank(G) = \crank(H) = \crank(K)$.  In particular if the kernel $K$ is empty then $\crank(G)=0$.

Let $G$ be a subcubic pseudograph, with non-empty kernel $K$. Let $v_2(K)$ be the number of singleton components of $K$ consisting of a vertex with a loop, which is the number of components of $G$ with exactly one cycle. Let $v_3(K)$ be the number of vertices of degree 3 in $K$, which is at most the number of vertices of degree 3 in $G$.  Each component of $K$ containing a vertex of degree 3 is cubic. Note that $v_3(K) + v_2(K) = v(K)$.
%
To upper bound $\crank(K)$ (and thus $\crank(G)$), we consider separately the $v_2(K)$ singleton components of $K$ and the $\kappa(K) - v_2(K) \leqslant \frac12 v_3(K)$ cubic components, and see that
\begin{eqnarray*}
\crank(K) & = &
e(K)-v(K)+ \kappa(K)\\
&=&
v_2(K) +
(\tfrac32 v_3(K) - v_3(K) + \kappa(K) - v_2(K)) \\
& = & v_2(K) +\tfrac12 v_3(K) +(\kappa(K)-v_2(K)) \\
& \leqslant &
v_2(K) + v_3(K) \;\; = \;\; v(K)\,.
\end{eqnarray*}
Thus
\begin{equation} \label{eqn.kernel}
 \crank(G) =\crank(K) \leqslant  v(K) \,.   
\end{equation}  
This result is best possible: $\crank(K) = v(K)$ if and only if each component of $K$ is either a singleton vertex with a loop 
or consists of two vertices joined by three parallel edges. We can now prove Lemma~\ref{lem.XsubsetheredF}.

\begin{proof}[Proof of Lemma~\ref{lem.XsubsetheredF}]
Recall that $k = k(n) = 2 \lfloor n g / (2n+3g) \rfloor$ and note that $k$ is even, $k < g$ and $k \sim g$ as $n \to \infty$.
Recall also that $s = s(n) = \lfloor 2 (n-k)/3k \rfloor$, so $s\sim 2n/\,3k \sim 2n/\,3g$ and $s \to \infty$ as $n \to \infty$.  We may assume that $n$ is sufficiently large that $s \geqslant n_0$.  
\smallskip

Let $G \in \cZ^k_n$.
Suppose for a contradiction that there is a nonempty set $W \subseteq [n]$ such that the induced subgraph $G[W]$ has $\crank(G[W]) >g(|W|)$; and we may suppose that the set $W$ is minimal with this property.  Now the kernel $K$ of $G[W]$ is nonempty (since otherwise $\crank(G[W])=0$) and in particular $G[W]$ has a cycle, so $|W| \geqslant 3s+3 \geqslant 3n_0+3 \geqslant n_0 +1$. 
Thus $g(n)$ is non-decreasing for $n \geqslant |W|-1$.  If $G[W]$ had a leaf $w \in W$ and $W'=W \setminus \{w\}$, then
\[ \crank(G[W']) = \crank(G[W]) > g(|W|) \geqslant g(W') \]
contradicting the minimality of $W$.  Hence
$G[W]$ has no leaves, and so all vertices have degree 2 or 3. 

The number of components of $G[W]$ with exactly one cycle is $v_2(K)$; and each such component contains at least $3s+3$ vertices.  Consider now the components of $G[W]$ which correspond to cubic components of~$K$. A loop at a vertex $u$ in a cubic component of $K$ corresponds to at least $3s+2$ vertices of degree 2 in $G[W]$. 
A non-loop edge $uv$ in a cubic component of $K$ corresponds to at least $s$ vertices of degree 2 in $G[W]$. 
If there are $x$ loops in the cubic components of $K$ then there are $\tfrac32 v_3(K) -x$ non-loop edges; and thus the total number of vertices in the components of $G[W]$ which correspond to cubic components of $K$ is at least \[ x(3s+2) + (\tfrac32 v_3(K) -x) s + v_3(K) = (\tfrac32 s +1) v_3(K)  + x(2s+2) \geqslant (\tfrac32 s +1) v_3(K)\,. \]
Hence
\[ |W| \geqslant (3s+3)v_2(K) + (\tfrac32 s +1) v_3(K) \geqslant (\tfrac32 s +1) v(K)\,. \]
%
Thus by~(\ref{eqn.kernel}) 
we have
\[\crank(G[W]) = \crank(K) \leqslant v(K) \leqslant \tfrac{2\,|W|}{2+ 3 s}\,. \]
We shall obtain the desired contradiction by showing that
$\tfrac{2\,|W|}{2+3s} \leqslant g(|W|)$, that is $2+3s \geqslant 2\,|W|/g(|W|)$.  To show this, since $n/g(n)$ is non-decreasing for $n \geqslant n_0$ and $|W| \geqslant n_0$, it suffices to show that $2+3s \geqslant 2n/g$ (still writing $g$ for $g(n)$), that is $s \geqslant \frac{2n-2g}{3g}$.  By the definition of $s$, this must hold if
\[\frac{2(n-k)}{3k} -1 \geqslant \frac{2n-2g}{3g}\,.\]
But this inequality simplifies to
\[  k (2n+3g) \leqslant 2ng\,,\]
which follows immediately from the definition of $k$.
This completes the proof.
\end{proof}

We continue by proving the following lemma, showing that $\cZ^k_n$ is large.
\begin{lemma}
\label{thm.excessagain}
Let the function $k=k(n)$ take even integer values and satisfy $k(n) \leqslant n$ and $k(n)\rightarrow \infty$ as $n\rightarrow \infty$.  Let $c= \tfrac16 (3/e)^{3/2} \; ( \approx 0.1932)$. Then the class $\cZ^k$ satisfies 
\begin{equation} 
|\cZ^k_n|\geqslant 
n! \, (e c+o(1))^k \, k^{k/2}
\end{equation}
\end{lemma}
\begin{proof}
For even $k$, the number $C(k)$ of cubic graphs 
on $[k]$ satisfies $ C(k) \sim (2/e)^{1/2} c^k k^{3k/2}$ as $k \to \infty$, see for example Corollary 9.8 of~\cite{JLR}.
We may assume that $n$ is sufficiently large that $s \geqslant n_0$.
We construct graphs in $\cZ^k_n$ by picking a $k$-set $U \subseteq [n]$ and a cubic graph $G_0$ on $U$ (so $G_0$ has $\tfrac32 k$ edges), and using the $n-k$ vertices in $\overline{U} = [n] \backslash U$ to subdivide each edge of $G_0$ at least $s$ times.  (This is possible since $n-k \geqslant (\tfrac32 k)\, s$.) To count the graphs constructed, we may think of listing the edges of $G_0$ in lexicographic order, oriented away from the smaller end-vertex, and listing the vertices in $\overline{U}$ in any one of the $(n-k)!$ possible orders; then inserting the first $s$ vertices of $\overline{U}$ in order in the first oriented edge, the next $s$ vertices in the next edge, and so on, until we insert the remaining at least $s$ vertices in the last edge.  In this way each graph is constructed just once. Thus
\begin{eqnarray*}
| \cZ^k_n| & \geqslant &\binom{n}{k} \, C(k) \, (n-k)! \; = \; n! \, C(k)/k!\\
& = &
  n! \, (c+o(1))^k k^{3k/2}/k! \\
& = &  
  n! \, (e c+o(1))^k \, k^{k/2} 
\end{eqnarray*}
since $k! = ((1+o(1))(k/e))^k$.
\end{proof}
We can now complete the proof of Lemma~\ref{lem.excessnew}.
\begin{proof}[Proof of Lemma~\ref{lem.excessnew}]
Let $k = k(n) = 2 \lfloor n g / (2n+3g) \rfloor$.
Combining Lemmas~\ref{lem.XsubsetheredF} and~\ref{thm.excessagain}, for $n$ sufficiently large we have
\begin{eqnarray*}
|\hered(\cF^g)| & \geqslant &| \cZ^k_n|\;\;\;\mbox{ by Lemma~\ref{lem.XsubsetheredF} }\\
& = &  
  n! \, (e c+o(1))^k \, k^{k/2} \;\;\;
  \mbox{ by Lemma~\ref{thm.excessagain} }\\
& = &
n! \, (e c+o(1))^g g^{(1+o(1))g/2}\\
& = &
n! \, g^{(1+o(1))g/2},
\end{eqnarray*}
as required.
\end{proof}
\bigskip

\noindent
\emph{Lower bounding the size of  $\hered(\cA^g)$}

We have already noted that $\hered(\cF^g)\subseteq \hered(\cA^g)$, so the result corresponding to Theorem~\ref{cor.excess} for $\hered(\cA^g)$ follows directly from Theorem~\ref{cor.excess}. However,  we can obtain an improved explicit bound for the size of $\hered(\cA^g)$ compared to that in Lemma~\ref{lem.excessnew} (where the lower bound was $n!\, g^{(1+o(1))g/2}$). We state this in the following lemma.

\begin{lemma}
\label{thm.excessA}
Let the genus function $g$ satisfy $g(n) \to \infty$ and $g(n)/n \to 0$ as $n \to \infty$; and suppose that there is an $n_0$ such that for $n \geqslant n_0$, $g(n)$ is non-decreasing and $g(n)/n$ is non-increasing.  
Then 
\begin{equation} 
|\hered\,(\cA^g)_n|\geqslant 
n!\: g^{(1+o(1))\,g}\,.
\end{equation}
\end{lemma}

Our lower bound approach for proving Lemma~\ref{thm.excessA} follows the pattern of the proof of Lemma~\ref{lem.excessnew} except that it involves the `excess' of a graph rather than the cycle rank.  The \emph{excess} $\xs(G)$ of a graph G is the sum over its non-tree components $C$ of $e(C)-v(C)$. 
Thus $\xs(G) \geqslant 0$, and $\xs(G) = 0$ if and only if each component has at most one cycle (that is, there are no `complex' components).
Also, deleting a leaf or subdividing an edge does not change the excess.
Observe that for a graph $G$ with $\kappa^-$ non-tree components, the cycle rank $\crank(G)$ satisfies $\crank(G) = \xs(G)+ \kappa^-$, see Section~\ref{subsec.proof2b}. It is more convenient here to work with $\xs(G)$ rather than $\crank(G)$, since we will need to consider subgraphs that may fail to be connected. 

Given a genus function $g$ we let $\cX\cS^g$ be the class of all graphs $G$ with $\xs(G) \leqslant g(n)$ where $n=v(G)$. 
We show in Lemma~\ref{lem.xs} that $\cX\cS^g \subseteq \cA^g$ and so  $\hered\,(\cX\cS^g) \subseteq \hered\,(\cA^g)$.
We then show in Lemma~\ref{thm.excess} that $\cZ^k \subseteq \hered(\cX\cS^g)$ for a suitable choice of $k \sim 2g$ (previously we had $k \sim g$), and using Lemma~\ref{thm.excessagain} we show that $\cZ^k$ is suitably large.

%
To prove Lemma~\ref{lem.xs} we use one preliminary lemma.
\begin{lemma} \label{lem.xs0}
Every graph $G$ with $\xs(G) \geqslant 1$ has a rotation system with at least 3 faces.
\end{lemma}
%
\begin{proof}
Let $\xs(G) \geqslant 1$.  By considering the core of $G$, we may see that it suffices to assume that each vertex degree is at least~2. If $G$ contains two cycles sharing at most one edge, then clearly there is a rotation system for $G$ such that both cycles form facial walks, and so in total there must be at least 3 facial walks, as required.  Similarly, there is a rotation system as desired if the two cycles intersect in a subdivided edge, that is, in a path in which each internal vertex has degree 2. 

Suppose that in $G$ there are no two edge-disjoint cycles. We claim that there must be two cycles which intersect in an edge or subdivided edge. Let us check first that there are two cycles which intersect (exactly) in a path.  To see this, let $C_1, C_2$ be any two distinct cycles: then part of $C_2$ forms a path $P$ with no internal vertices in $C_1$ between distinct vertices $u, v$ in $C_1\,$; and this path $P$ together with either one of the two parts of $C_1$ joining $u,v$ form a cycle $C_3$ which intersects $C_1$ exactly in the path $P$.  

Now let $C_1$ and $C_2$ be cycles which intersect in a shortest possible path $P$.  We want to show that each internal vertex in $P$ has degree 2.  Suppose for a contradiction that some internal vertex $v$ in $P$ is incident to 
an edge $vw$ not in $P$.
Start walking from $v$ along $vw$ and continue (always picking a new edge) until we first meet a vertex $z$ in $C_1$ or $C_2$.  Since there are no two edge-disjoint cycles we must form a path $Q$ (with all vertices distinct) and the final vertex $z$ of $Q$ is not $v$ and indeed is not in $P$ (by the minimality of $P$).
Suppose wlog that $z$ is in $C_2$. Then the distinct vertices $v$ and $z$ divide $C_2$ into two parts. Pick one of these parts, and form the cycle $C'_2$ from this part and the path $Q$. Then $C_1$ and $C'_2$ intersect in a path strictly contained in $P$.  But this contradicts our choice of $C_1$ and $C_2$, and thus completes the proof of the lemma.
\end{proof}

\begin{lemma} \label{lem.xs}
For every graph $G$, if $\xs(G)=h$ then $G \in \cO\cE^{h} \cap \cN\cE^{h}$; that is, $G$ has an embedding in an orientable and a non-orientable surface of Euler genus at most $\xs(G)$.
\end{lemma}
\begin{proof}
Consider a nonplanar component $C$ of a graph $G$ (the result clearly holds for planar graphs $G$); and note that $\xs(C) \geqslant 1$ (indeed we have $\xs(C) \geqslant 3$).
By Lemma~\ref{lem.xs0} there is a rotation system for $C$ with $f \geqslant 3$ faces.  By Euler's formula, the corresponding embedding has Euler genus $e-v-f+2 \leqslant \xs(C) -1$.  It follows that the union of the non-tree components of $G$ has an embedding in an orientable surface of Euler genus at most $\xs(G)-1$; and extending the embedding to include any tree components, we see that $G$ has such an embedding $\phi$. 
Finally, by Observation~\ref{nonor_from_or}, $G$ must have a non-orientable embedding with Euler genus at most $\xs(G)$. 
\end{proof}
We could shorten the proof of Lemma~\ref{lem.xs} (essentially omitting Lemma~\ref{lem.xs0}) if we were willing to replace $\cN\cE^h$ by $\cN\cE^{h+1}$, but we have chosen to be tidy.
By Lemma~\ref{lem.xs} we have $\cX\cS^g \subseteq \cA^g$ and so  $\hered\,(\cX\cS^g) \subseteq \hered\,(\cA^g)$.  This gives the first inequality in the conclusion~(\ref{eqn.XS}) of the next lemma.
\begin{lemma}
\label{thm.excess}
Let the genus function $g$
satisfy $g(n) \to \infty$ and $g(n)/n \to 0$ as $n \to \infty$; and suppose that there is an $n_0$ 
such that for $n \geqslant n_0$, $g(n)$ is non-decreasing and $g(n)/n$ is non-increasing.  Then
\begin{equation} \label{eqn.XS}
|\hered\,(\cA^g)_n|\geqslant |\hered\,(\cX\cS^g)_n|\geqslant 
n!\: g^{(1+o(1))\,g}\,.
\end{equation}
\end{lemma}
\begin{proof}[Proof of Lemma~\ref{thm.excess}]
We shall continue often to write $g$ for $g(n)$.  The idea of the proof is as follows. If the graph $G$ is a subdivision of a $k$-vertex cubic graph $H$ then
\begin{equation} \label{eqn.excess0}
\xs(G) = \xs(H)
  = \tfrac12 k \,.
\end{equation}
Thus if $G$ has $n$ vertices and $k \leqslant 2 \, g(n)$ then $G \in \cX\cS^g$.
If each edge of the original cubic graph $H$ was subdivided sufficiently often, and we introduced a little slack, then in fact $G \in \hered(\cX\cS^g)$.  Since there are many choices for $G$ we can deduce that $\hered(\cX\cS^g)$ is large.
\smallskip

Now for the details.
As before, let $\cZ^k_n$ be the set of graphs on $[n]$ which are subdivisions of a $k$-vertex cubic graph, and such that the distance between any two vertices of degree 3 is at least $s+1$ (and so the girth is at least $3s+3$). This time, we choose an even integer $k=k(n)$ a little less than $2g$: we let
\[k= 2 \lfloor \frac{n g}{n+ 3g} \rfloor\]
(so $k \sim 2g$, and $k \to \infty$ as $n \to \infty$).
Recall that $s = s(n) = \lfloor 2 (n-k)/3k \rfloor$, so $s\sim 2n/\,3k \sim n/\,3g$ and $s \to \infty$ as $n \to \infty$.  We may assume that $n$ is sufficiently large that $s \geqslant n_0$.
By Lemma~\ref{thm.excessagain}, 
\begin{eqnarray*}
| \cZ^k_n| & \geqslant &
n! \, (e c+o(1))^k \, k^{k/2}\\
& = &
n! \, ( 2 (e c)^2 +o(1))^g\, g^{(1+o(1))g}\,.
\end{eqnarray*}
But $2(ec)^2 \approx 0.55 > \tfrac12$, so $|\cZ^k_n| \geqslant n!\, (\tfrac12 g)^{(1+o(1))g}$. We shall complete the proof by showing that $\cZ^k_n \subseteq \hered(\cX\cS^g_n)$ (for $n$ sufficiently large). 

Let $G \in \cZ^k_n$.
Suppose for a contradiction that $G \not\in \hered(\cX\cS^g)_n$, so there is a nonempty set $W \subseteq [n]$ such that $\xs(G[W]) > g(|W|)$.  An acyclic graph has excess 0, so $|W| \geqslant 3s+3 \geqslant 3n_0+3$.
We may suppose that the set $W$ is minimal such that $\xs(G[W]) > g(|W|)$.  Then, since $g$ is non-decreasing for $n \geqslant n_0$ and $|W| \geqslant n_0 +1$, it follows that $G[W]$ has no leaves, and so all vertices have degree 2 or 3.  If all vertices had degree 2 then the excess would be 0, so there must vertices of degree 3, say $i \geqslant 2$ vertices of degree 3. For each vertex $v$ in $G[W]$ of degree~3, the three complete subdivided edges incident with $v$ must be in $G[W]$ (since there are no leaves).  Thus there are at least $i \cdot \tfrac32 s$ vertices of degree 2 (since each vertex of degree 2 is counted at most twice), and so $|W| \geqslant i \cdot (1+ \tfrac32 s)$.  Thus by~(\ref{eqn.excess0})
\[ \xs(G[W])  = \tfrac12 i \leqslant \tfrac{|W|}{2+3s} .\]
We shall obtain the desired contradiction by showing that
$\tfrac{|W|}{2+3s} \leqslant g(|W|)$, that is $2+3s \geqslant |W|/g(|W|)$.  To show this, since $n/g(n)$ is non-decreasing for $n \geqslant n_0$ (and $|W| \geqslant n_0$), it suffices to show that $2+3s \geqslant n/g$ (still writing $g$ for $g(n)$), that is $s \geqslant \frac{n-2g}{3g}$.  By the definition of $s$, this must hold if
\[\frac{2(n-k)}{3k} -1 \geqslant \frac{n-2g}{3g}\,,\]
which simplifies to
\begin{equation} 
    k (n+3g) \leqslant 2ng\,.
\end{equation}
But this inequality holds by the definition of $k$, so we have the desired contradiction.
This completes the proof of Lemma~\ref{thm.excess} (and thus of 
Theorem~\ref{cor.excess}).
\end{proof}

It remains only to prove Proposition~\ref{prop.heredsmall} to complete the proofs of our results on hereditarily embeddable graphs.
\begin{proof}[Proof of Proposition~\ref{prop.heredsmall}]
By inequality~(\ref{eqn.lbasymp}) we have
$|\cA^{g-2}_n| \ll |\cA^{g}_n|$.  Let $\cL$ be the class of graphs $G$ such that if $v(G)=n$ then $G$ has less than $2 \alpha n$ leaves.  Observe that $2 \alpha < \rho(\cP)$. Thus $|\cA^g_n \cap \cL_n| \ll |\cA^g_n|$ by Theorem 6 of~\cite{MSrandom} (which depends only on Theorem 1 in the present paper). 

Let $n \geqslant n_0$, and let $G \in (\cA^g \setminus (\cA^{g-2} \cup \cL))_n$.  It suffices to show that $G \not\in \hered(\cA^g)$.
Let $1 \leqslant k \leqslant 2 \alpha n$ be such that $g(n) \geqslant g(n\!-\!k)+2$. Observe that $G$ has at least $k$ leaves, since $G \not\in \cL_n$. Form $H$ by deleting $k$ leaves from $G$.  Since $g(n\!-\!k) \leqslant g(n)-2$ and $G \not\in \cA^{g-2}_n$ we have $G \not\in \cA^{g(n-k)}_n$.  Hence the $(n\!-\!k)$-vertex induced subgraph $H$ of $G$ is not in $\cA^{g(n-k)}$ (since we could add back the deleted leaves while keeping embedded in the same surface), and so $G \not\in \hered(\cA^g)$, as required.
\end{proof}
\subsection{The class $\strong(\cA^g)$ of certifiably hereditarily embeddable graphs}
\label{subsec.stronghered}
In the first part of this section, we investigated graph classes where a graph $G$ is in the class if and only if each (induced) subgraph of $G$ has an embedding in a suitable surface with sufficiently small Euler genus. We could be more demanding and insist that there must be a single cellular embedding of $G$ such that each induced embedding of an induced subgraph has sufficiently small Euler genus. 
It is natural here to focus on orientable surfaces.
Given a genus function $g=g(n)$, we say that a graph $G$ is \emph{certifiably hereditarily} in $\cO\cE^g$
if there is a cellular embedding of $G$ in some orientable surface such that for each nonempty set $W$ of vertices the induced embedding of $G[W]$ (which is orientable) has Euler genus at most $g(|W|)$. Let $\strong(\cO\cE^g)$ denote the class of such graphs. Then $\strong(\cO\cE^g) \subseteq \hered(\cO\cE^g)$, and this is typically a proper containment.
For example, if $g$ satisfies $g(n)=0$ for $n \leqslant 4$ and $g(5)=2$
then clearly $K_5 \in \hered(\cO\cE^g)$, and we will see below 
that $K_5 \not\in \strong(\cO\cE^g)$.
On the other hand, $\strong(\cO\cE^g) \supseteq \hered(\cF^g)$, since every orientable embedding of a graph $G$ in $\hered(\cF^g)$ certifies that $G$ is in $\strong(\cO\cE^g)$.
Thus by (\ref{eqn.rhopos-hered}) the threshold when $\rho$ drops to 0 for certifiably hereditarily embeddable graphs still occurs around $n/\log n$.

We have one loose end to tidy up here.
\begin{proof}[Proof that $K_5 \not\in \strong(\cO\cE^g)$ when $g(n)=0$ for $n \leqslant 4$]
\medskip

For $K_4$ on vertex set $[4]$, there is a unique rotation system which gives an embedding in the sphere $\bS_0$ and which has cyclic order $\pi(1)=(234)$ for vertex 1. 
The rest of the rotation system is $\pi(2)= (143)$, $\pi(3)=(124)$ and $\pi(4)=(132)$, and it is a triangulation. 

Now consider a rotation system $\pi$ for $K_5$ on $[5]$.  We want to show that for at least one vertex $i \in [5]$, the induced rotation system on $[5] \setminus \{i\}$ is nonplanar.  We may assume wlog that $\pi(1)=(2345)$.  Suppose for a contradiction that for each $i =2,..,5$ the induced rotation system on $[5] \setminus \{i\}$ is planar.
When we drop vertex 2, the induced cyclic order $\pi(1)$ on $\{3,4,5\}$
is $(345)$; and by the assumption that the induced embedding on $\{1,3,4,5\}$ is planar and the uniqueness of the planar embedding, we see that $\pi(3)$ contains the subsequence $(154)$, $\pi(4)$ contains $(135)$, and $\pi(5)$ contains $(143)$.  Arguing similarly when we drop other vertices, we see that the cyclic orders $\pi(i)$ must contain the subsequences shown:
\begin{center}
\[
\begin{array}{c|cccc}
   & \mbox{drop 2} & \mbox{drop 3} & \mbox{drop 4} & \mbox{drop 5} \\
  \pi(1) & 345 & 245 & 235 & 234 \\
  \pi(2) &  -  & 154 & 153 & 143 \\
  \pi(3) & 154 &  -  & 125 & 124 \\
  \pi(4) & 135 & 125 &  -  & 132 \\
  \pi(5) & 143 & 142 & 132 &  -  \\
\end{array}
\]
\end{center}
%
It follows from the table that $\pi(2)= (1543)$, $\pi(3)=(1254)$, $\pi(4)=(1325)$ and $\pi(5) =(1432)$.  But now, in the induced embedding on $\{2,3,4,5\}$, 
there is a facial walk with vertices $2,5,4,3,2$ of length 4; thus we do not have a triangulation, and so we do not have a planar embedding of the copy of $K_4$ on $\{2,3,4,5\}$.
\end{proof}

\section{Minor-closed classes $\minor(\cA^g)$ of embeddable graphs}
\label{sec.mc}
In this section we prove Theorem~\ref{thm.minor} on the graph class $\minor(\cA^g)$.
Recall that, 
given a genus function~$g$, $\minor(\cA^g)$ is the class of graphs $G$ such that,
for each $k =1,\ldots,v(G)$, each $k$-vertex minor $H$ of $G$ is in $\cA^g_k$;
%
and recall that $\minor(\cA^g) \supseteq \cP$.  By the Kuratowski-Wagner Theorem (see for example~\cite{BondyMurty,Diestel}), a graph $G$ is in $\cP$ if and only if it has no minor $K_5$ or $K_{3,3}$.
Thus for example $\minor(\cE^g)=\cP$ if and only if $g(5)=g(6)=0$, 
since $K_5$ and $K_{3,3}$ both
embed in each surface other than $\bS_0$ (in the orientable case note that $\minor(\cO\cE^g)=\cP$ if and only if $g(5)<2$ and $g(6)<2$).
At the other extreme, for each $n\in \N$ let $g^*(n)$ be the least $h \geqslant 0$ such that $K_n \in \cA^h$ 
: then $g^*(n) \sim \tfrac16 n^2$
(see near the end of Section~\ref{subsec.embed}, or below,  for exact values).
But $\minor(\cA^g)$ contains all graphs if and only if $K_n \in \cA^g$ for each $n \in \N$, and this happens if and only if $g(n) \geqslant g^*(n)$ for each $n \in \N$.

We are ready to prove Theorem~\ref{thm.minor}.  The proof of the first part is very short. The proof of the second part will show that for small $\eps>0$ we may take the constant $c=c(\eps)$ to be about $\frac13 \log_2 \frac1{\eps}$. 

\begin{proof}[Proof of Theorem~\ref{thm.minor}] 
For the first part, note that $\minor(\cA^g)$ is closed under taking minors, and for any class $\cB$ of graphs which is closed under minors and does not contain all graphs we have $\rho(\cB) \geqslant \tilde{\rho}(\tilde{\cB}) >0$, see~\cite{nstw2006,AFS2009,dn2010}.  Hence either $\minor(\cA^g)$ contains all graphs, which happens if and only if $g \geqslant g^*$;
or $\tilde{\rho}(\minor(\tilde{\cA}^g)) >0$, and so $\rho(\minor(\cA^g))>0$.  (Thus the threshold when the radius of convergence drops to $0$ occurs when $g(n) \sim n^2/6$.)

Now consider the second part of the theorem.
Let $\eps>0$, and fix a large $t \in \N$.  Let $n \geqslant t$ and construct graphs on $[n]$ as follows.  Partition $[n]$ into $k = \lfloor n/t \rfloor$ parts of size $t$, with an extra part of size $u \leqslant t-1$ if $t$ does not divide $n$. If $t|n$ (so there is no extra part) we set $u=0$. Choose a vertex in each part (say the smallest vertex).  Pick an order on the $k$ or $k+1$ chosen vertices, list the vertices as $v_1, v_2, \ldots$ and add the edges $v_i v_{i+1}$. We obtain at least $\frac12 k!$ unoriented paths.  Put an arbitrary connected graph on each part.  We have  $2^{(\frac12 +o(1)) t^2}$ choices for each part of size $t$ (where $o(1)$ is as $t$ gets large), and if there is an extra part of size $u$ 
then we have at least $u!$ choices for this part.  In total we make at least
\[ \frac{n!}{k! \, (t!)^k \, u!} \,  \tfrac12 k! \, u! \, ( 2^{(\frac12 +o(1)) t^2} )^ k\]
constructions (recall that $0! =1$), and each graph is constructed at most once.  So if $t$ is chosen sufficiently large, the number of distinct graphs constructed is at least
\[ n! \, ( 2^{(\frac12 +o(1)) t^2} / t! )^ k  \geqslant  n! \, ( 2^{(\frac12 +o(1)) t^2} )^ {n/t -1}  =  n! \,  ( 2^{(\frac12 +o(1)) t} )^ n \geqslant n! \, \eps^{-n} \]
for $n$ sufficiently large.

Recall that, for each $n \geqslant 3$, in the orientable case (when $g^*(n)$ is the least $h$ such that $K_n \in \cO\cE^h$) we have $g^*(n)=2\left\lceil \tfrac1{12}(n-3)(n-4) \right\rceil$;
and $g^*(n) = \left\lceil \tfrac16(n-3)(n-4) \right\rceil$ 
in the non-orientable case, except that $g^*(7)=3$.
Thus, for both the orientable and non-orientable cases, for each $1 \leqslant n \neq 6$
\[ g^*(n+1) \leqslant 2\left\lceil \tfrac1{12}(n-2)(n-3) \right\rceil \leqslant \tfrac16\big((n-2)(n-3)+10\big) = \tfrac16(n^2 -5n +16)\,, \]
where the second inequality holds since $(n-2)(n-3)$ is always an even integer.
Thus $g^*(n+1) \leqslant \tfrac16n^2$ for each $n \geqslant 4$, including $n=6$. But $g^*(n+1)$ is $0$ for $n =0, 1 ,2$ and $3$; and
so $g^*(n+1) \leqslant \tfrac{1}{6}n^2$ for each $n \geqslant 1$.  
\smallskip

\smallskip

Now consider one of the graphs $G$ constructed on $[n]$, and a minor $H$ of $G$ with $s$ vertices. Each vertex $w$ of $H$ corresponds to a connected subgraph $H_w$ of $G$, where these subgraphs are vertex-disjoint. 
Consider the $i$th part of $G$, with chosen vertex $v^*_i$.
Suppose that $H$ contains $a_i$ vertices $w$ corresponding to connected subgraphs $H_w$ of $G$ which are contained within the $i$th part of $G$ and do not contain $v^*_i$, and so are
completely contained within the $i$th part.  There may also be a vertex of $H$ corresponding to a connected subgraph of $G$ which contains $v^*_i$ and perhaps other vertices of the $i$th or other parts of $G$.
Then each $a_i \leqslant t-1$
and $\sum_i a_i \leqslant s$; and $H$ can be embedded in a surface of Euler genus at most
\[ \sum_i g^*(a_i+1) \leqslant  \tfrac16 \sum_i {a_i}^2  \leqslant  \tfrac16\, \frac{s}{t-1}\,  (t-1)^2 = \tfrac16 (t-1) \, s.\]
Thus if we set $c=\lceil \tfrac16 (t-1) \rceil$ and $g(n) = c n$ then $G$ is in $\minor(\cA^g)$.  Hence $|\minor(\cA^g_n)| \geqslant n!\, \eps^{-n}$ for $n$ sufficiently large, and $\rho(\minor(\cA^g)) \leqslant \eps$, as required.
\end{proof}

Interesting questions on minor-closed classes remain open. For example, we saw that $\rho(\minor(\cA^g))$ is arbitrarily small for a large linear function~$g$.  But do we need $g$ to be so large?  Given $\eps>0$, is there a constant $c=c(\eps)$ such that setting $g(n)=n+c$
we have $\rho(\minor(\cA^g))<\eps$?
\smallskip

Finally here let us briefly consider topological minors.  
A graph $H$ is a \emph{topological minor} of a graph $G$ if $H$ can be obtained from a subgraph of $G$ by a sequence of edge-contractions where each edge is incident to a vertex of degree 2, see for example~\cite{Diestel}. 
Given a class $\cB$ of graphs, let $\tminor(\cB)$ be the class of graphs $G$ such that each topological minor of $G$ is in $\cB$.  We call $\tminor(\cB)$ the topological-minor-closed part of $\cB$.  Of course we always have $\cP \subseteq \minor(\cA^g) \subseteq \tminor(\cA^g) \subseteq \cA^g$, and so in particular $\rho(\cP) \geqslant \rho(\tminor(\cA^g))$.  

Let us restrict our attention here to $\cE^g$ (rather than $\cA^g$). As with (usual) minors, we have $\tminor(\cE^g)=\cP$ if and only if $g(5)=g(6)=0$. 
However, in other ways the behaviour is very different from that of minors, and in particular there is no result like Theorem~\ref{thm.minor}.
For example, define a genus function $g$ by setting $g(n)=0$ for $n \leqslant 5$ and $g(n) = \lfloor \tfrac12 n \rfloor$ for $n \geqslant 6$.
Then clearly $K_5 \not\in \cE^g$ since $g(5)=0$.  But each subcubic graph $G$ (with each degree at most 3) has $\xs(G) \leqslant \tfrac12 v(G)$.  Hence, noting that each subcubic graph on at most 5 vertices is planar, we have $G \in \cE^g$ by Lemma~\ref{lem.xs}. Also, each $\tminor$ of a subcubic graph is subcubic, so each subcubic graph is in $\tminor(\cE^g)$; and it follows that $\rho(\tminor(\cE^g))=0$.
See the recent paper~\cite{CLS2019} for more information and results related to this topic.


\section{Concluding remarks and questions}
\label{sec.concl}
As earlier, let $g$ be a genus function and let $\cA^g$ denote one of the graph classes $\cE^g$, $\cO\cE^g$, $\cN\cE^g$ or $\cO\cE^g \cap \cN\cE^g$.
We have given estimates and bounds on the sizes of the sets $\cA_n^g$, where for example $\cE_n^g$
is the set of graphs on vertex set $[n]$ embeddable in 
a surface of Euler genus at most $g(n)$;
and we have given some corresponding results for the hereditary classes $\hered(\cA^g)$ and $\strong(\cA^g)$, the minor-closed class $\minor(\cA^g)$, the topological-minor-closed class $\tminor(\cA^g)$,
and for related unlabelled graph classes.
Some of these results will be used in the companion paper~\cite{MSrandom} where we investigate random graphs sampled uniformly from such classes.
Many interesting questions remain open concerning the sizes of these classes of graphs. We focus in this concluding section on whether the class $\cA^g$ has a growth constant $\gamma$ and if so whether $\gamma=\gamma_\cP$.

We have seen that (from a distance) the graph class $\cA^g$ is `similar' in size to $\cP$ for a `small' genus function $g$, 
and much bigger for a `large' $g$.
Can we pin this down more precisely? Theorem~\ref{thm.gc-estimate} (a) shows that $\cA^g$ has growth constant $\gamma_{\cP}$ as long as $g(n)=o(n/\log^3n)$.  Also we saw from Theorem~\ref{thm.lowerbound} (a) that if $\cA^g$ has growth constant $\gamma_{\cP}$ then $g(n)=o(n/\log n)$. Perhaps the converse holds? 

\begin{conjecture} \label{conj.gcP}
  $\cA^g$ has growth constant $\gamma_{\cP}$ if and only if $\, g(n)=o(n/\log n)$.
\end{conjecture}

We saw in~(\ref{eqn.rhopos}) and~(\ref{eqn.rhotpos}) that, in both the labelled and the unlabelled cases, the radius of convergence is strictly positive if and only if $g(n) = O(n/\log n)$.  In the labelled case, for suitably well behaved genus functions~$g$, 
perhaps we have a growth constant whenever we have a strictly positive radius of convergence?
\begin{conjecture} \label{conj.gc}
If $c>0$ is a constant and $g(n) \sim c n/\log n$, then  $\cA^g$ has a growth constant $\gamma = \gamma(c)$.
\end{conjecture}
\noindent
Suppose temporarily that the growth constants $\gamma(c)$ exist as in Conjecture~\ref{conj.gc}. Then inequality (\ref{eqn.lbasymp}) shows that  $\gamma(c)$ is strictly increasing as a function of $c$, and by
Theorem~\ref{thm.gc-estimate} (b) we have $\gamma(c) \to \infty$ as $c \to \infty$.
Also, $\gamma(c) > \gamma_{\cP}$ for each $c > 0$. Does $\gamma(c) \to \gamma_{\cP}$ as $c \to 0$?

Now let us briefly consider unlabelled graph classes.
As we remarked earlier, the set $\tP$ of unlabelled planar graphs has growth constant $\tilde{\gamma}_{\tP}$ where $\gamma_{\cP} < \tilde{\gamma}_{\tP} \leqslant 32.2$, see~\cite{RandomPlanar}. Further, for any fixed genus $h$, the set $\tA^h$ has the same growth constant $\tilde{\gamma}_{\tP}$, see \cite{graphsSurfaces}.
What can we say about the existence of a growth constant for $\tA^g$ for a non-constant genus function $g(n)$?


\bigskip

\noindent
\emph{Acknowledgement}

We would like to thank Mihyun Kang for helpful conversations.

\bibliographystyle{abbrv}
\bibliography{main}

\end{document}